\documentclass[a4paper]{article}
\usepackage{amssymb, amsthm, amsmath}
\usepackage{graphicx}
\usepackage{color}
\usepackage{dsfont}
\usepackage{bm}
\usepackage{pgf,tikz}
\usepackage{hyperref}
\usepackage{authblk}
\usepackage{xcolor}
\usetikzlibrary{arrows}

\newtheorem{theorem}{Theorem}[section]
\newtheorem{proposition}[theorem]{Proposition}
\newtheorem{lemma}[theorem]{Lemma}
\newtheorem{corollary}[theorem]{Corollary}

\theoremstyle{definition}
\newtheorem{definition}[theorem]{Definition}
\newtheorem{assumption}[theorem]{Assumption}
\newtheorem{remark}[theorem]{Remark}
\newtheorem{example}[theorem]{Example}

\newcommand{\R}{\mathbb{R}}
\newcommand{\Z}{\mathbb{Z}}
\newcommand{\E}{\mathbb{E}}

\renewcommand{\S}{\mathbb{S}}
\renewcommand{\P}{\mathbb{P}}

\newcommand{\FF}{\mathbb{F}}
\newcommand{\GG}{\mathbb{G}}

\newcommand{\EP}{ {\mathbb E}^{\mathbb{P}}}

\newcommand{\Filt}{\mathcal{F}}
\newcommand{\Gfilt}{\mathcal{G}}

\newcommand{\Mset}{\mathcal{M}}
\newcommand{\Nset}{\mathcal{N}}

\newcommand{\Borel}{\mathcal{B}}

\renewcommand{\AA}{\mathcal{A}}
\newcommand{\BB}{\mathcal{B}}
\newcommand{\XX}{\mathcal{X}}
\newcommand{\YY}{\mathcal{Y}}

\newcommand{\Pset}{\mathcal{P}}

\newcommand{\VV}{\mathcal{V}}

\let\inf\relax \DeclareMathOperator*\inf{\vphantom{p}inf}

\newcommand\I{\mathds{1}}

\newcommand{\norm}[1]{\left\|#1\right\|}

\newcommand{\as}{\text{a.s.}}

\newcommand{\tr}{\operatorname{tr}}
\newcommand{\Dt}{\mathcal{D}_t}
\newcommand{\Dx}{\nabla_x}
\newcommand{\Dxx}{\nabla_x^2}

\newcommand{\red}[1]{{\color{red}#1}}

\renewcommand{\red}[1]{}

\newcommand{\greg}[1]{{\color{blue}#1}}

\renewcommand{\greg}[1]{}

\title{Path Dependent Optimal Transport and Model Calibration on Exotic Derivatives}
\author[1,2]{Ivan Guo}
\author[1,2]{Gr\'egoire Loeper}

\affil[1]{School of Mathematics, Clayton Campus, Monash University, VIC, 3800, Australia}
\affil[2]{Centre for Quantitative Finance and Investment Strategies\thanks{\textbf{Acknowledgements\quad} The Centre for Quantitative Finance and Investment Strategies has been supported by BNP Paribas.
I. Guo has been partially supported by the Australian Research Council (Grant DP170101227).
The author would also like to thank Ben Goldys and the anonymous referee for their valuable comments and suggestions.
}, Monash University, Australia}

\begin{document}
\maketitle

\begin{abstract}
In this paper, we introduce and develop the theory of semimartingale optimal transport in a path dependent setting. Instead of the classical constraints on marginal distributions, we consider a general framework of path dependent constraints. Duality results are established, representing the solution in terms of path dependent partial differential equations (PPDEs). Moreover, we provide a dimension reduction result based on the new notion of ``semifiltrations'', which identifies appropriate Markovian state variables based on the constraints and the cost function. Our technique is then applied to the exact calibration of volatility models to the prices of general path dependent derivatives.
  
\bigskip
\noindent\textbf{Mathematics Subject Classification (2010):} 60H30,  91G80,  93E20

\noindent\textbf{Keywords:} optimal transport, path dependent PDE, volatility calibration
\end{abstract}

\section{Introduction}

Inspired by the seminal work on optimal transport by Benamou and Brenier \cite{benamou2000computational}, and the duality theory developed in \cite{Br5} and later \cite{loeper2006reconstruction}, we examine the problem of optimal transport by semimartingales in continuous time settings. The {\it semimartingale optimal transport problem} with constraints on marginals at given times has been studied by Tan and Touzi in \cite{tan2013optimal}, extending the work of Mikami and Thieullen \cite{mikami2006duality}. Other related works include \cite{huesmann2019benamou,veraguas2017martingale}. The main goal of our study is to extend this work by considering a much wider range of constraints. As shown in our abstract formulation, the constraints considered can be defined by any arbitrary closed and convex set of probability measures. In particular, this encompasses the  constraints on marginals of the classical optimal transport problem. Furthermore, it can include constraints such as bounds on the distributions as well as expectations of path dependent functions.

One of the outcomes of the duality techniques developed in \cite{Br5, loeper2006reconstruction} is that it bypasses the need to establish the dynamic programming principle, as the Hamilton-Jacobi-Bellman equation arises directly from the dual formulation. Our study establishes a natural connection between this optimal transport problem and the recent theory of path dependent partial differential equations (PPDEs) as developed in \cite{buckdahn2015pathwise,cont2013functional,dupire2009functional,ekren2016viscosity1,ekren2016viscosity2,ren2017comparison}. We also provide a formalisation of the dimension reduction technique used in numerical methods for BSDEs (see e.g., \cite{gobet2005regression,zhang2004numerical}) and path-dependent stochastic control problems (see e.g., \cite{tan2014discrete}) that reduces the complexity of the PPDE into a more tractable PDE by identifying relevant state variables from the cost function and the constraints. This is achieved via the introduction of \emph{semifiltrations}, which facilitates measurability results for both the solution of the optimal transport problem as well as the optimal drift and diffusion characteristics.

Recently, optimal transport has found applications in mathematical finance, particularly in the areas of robust hedging. For example, it is used to obtain model-free bounds for exotics derivatives in \cite{LabTouz} and robust hedging strategies in \cite{dolinsky2014martingale, hou2018robust}. As shown in Remark \ref{remrobusthedging}, the so-called robust pricing hedging duality in continuous time is also a corollary of our duality results. Connections between optimal transport and stochastic portfolio theory \cite{pal2018exponentially} have also been observed. 
In this paper, we focus on applications to model calibration, which is a crucial problem in financial modelling. 
The celebrated Dupire's formula \cite{dupire1994pricing} provides a unique way to recover a local volatility model from the knowledge of vanilla options for all strikes and maturities. However, it requires some form of price interpolation as only a finite number of options are available. Moreover, calibration to more sophisticated (path dependent) products cannot be achieved through this method. Beyond this analytical result, there are few theoretical advances that address the problem of calibration.
 Practitioners therefore often rely on parametric models, which they fine-tune to match observable instruments as much as possible.

The duality theory developed in this paper is used to exactly calibrate to any number of path dependent derivatives, without the need to perform any price interpolation. The idea of using optimal transport to calibrate local volatility models to European options was explored in \cite{guo2019local}, which is an adaptation of the numerical method of \cite{benamou2000computational}. A similar numerical algorithm for discrete option prices was also studied in \cite{avellaneda1997calibrating} in the context of entropy minimisation. In this paper, we extend the approach to the calibration of path dependent derivatives, resulting in a path dependent volatility function, a notion also explored in \cite{guyon2014path}. Despite the complex path dependent nature of the problem, efficient numerical algorithms are still possible via our dimension reduction result, where the optimal volatility function is Markovian in the state variables driving the derivatives. 
This is somewhat analogous to the fact that that European option prices can always be calibrated to local volatility models.
Our method can also be used to refine stochastic volatility models to exactly match a set of given payoffs, while remaining close to a reference stochastic volatility model. This therefore is a generalisation of the calibration of so-called local-stochastic volatility (LSV) models (see \cite{AbergelTachet, guo2019calibration} and the references therein). Recently, our duality results have also been used in \cite{guo2020joint} for the joint calibration of stock and VIX options, which has been an extremely difficulty problem eluding researchers for many years.

The paper is organised as follows. Section 2 introduces the basic notations used throughout the paper as well as the abstract formulation of the problem, including some examples. Section 3 contains the main results and their proofs. A summary of the main results is found in Theorem \ref{thmsummary01}. Section 4 includes the application of our results to volatility calibration, demonstrating numerical examples that calibrate to a large number of European, barrier and lookback options.

\section{Optimal transport under path dependent constraints}\label{sec2}

\subsection{Preliminaries}\label{secprelim}

Let $\Omega:=C([0,1];\R^d)$ be the set of continuous paths, $X$ be the canonical process and $\FF=(\Filt_t)_{0\leq t\leq 1}$ be the canonical filtration generated by $X$. For each $t\in[0,1]$, let $\Omega_t:=\{\omega_{\cdot\wedge t}:\omega\in\Omega\}$ be the set of paths stopped at time $t$ and  let $\Lambda:= \{(t,\omega): t\in[0,1], \omega\in \Omega_t\}$.
Next, let $\Omega^t := \{\omega\in \Omega : \omega_s=0, s\in[0,t]\}$ and $\Lambda^t = \{(u,\omega_{\cdot\wedge u}):u\in[t,1],\ \omega\in \Omega^t\}$ be the set of paths that stay at $0$ on $[0,t]$. 
The spaces $\Omega$, $\Omega_t$ and $\Omega^t$ are equipped the with the norm $\norm{\omega}_\infty=\sup_{t\in[0,1]} |\omega_t|$, while the spaces $\Lambda$ and $\Lambda^t$ are 
equipped with the metric $d_\infty((t,\omega), (t',\omega'))=|t-t'|+\norm{\omega_{\cdot\wedge t'}-\omega'_{\cdot\wedge t}}_\infty$.
Note that we have the relation $\Omega=\Omega_t+\Omega^t$. This also induces a natural isomorphism between $\Omega$ and the product $\Omega\cong \Omega_t \times\Omega^t$, as well as their $\sigma$-algebras, $\Filt_1 \cong \Filt_t \otimes \Borel(\Omega^t)$.


Given a Polish space $\XX$ equipped with its Borel $\sigma$-algebra, let $C(\XX)$ be the set of continuous functions on $\XX$, $C_b(\XX)$ be the set of bounded continuous functions and $\Mset(\XX)$ be the set of signed finite Borel measures on $\XX$.
On $C_b(\XX)$, let $\mathcal{T}_k$ denote the topology of uniform convergence on compact sets of $\XX$. Denote by $\mathcal{T}_t$ the finest locally convex topology on $C_b(\XX)$ which agrees with $\mathcal{T}_k$ on closed balls of $C_b(\XX)$ (via the uniform norm). The topology $\mathcal{T}_t$ was introduced by Le Cam \cite{lecam1957convergence} and is also known as the ``mixed topology'' \cite{fremlin1972bounded} or the ``substrict topology'' \cite{sentilles1972bounded}. For this paper, we will make use the following key result (see, e.g., \cite{fremlin1972bounded,sentilles1972bounded}).
\begin{proposition}
The $\mathcal{T}_t$ dual of $C_b(\XX)$ can be identified with $C_b(\XX)^*=\Mset(\XX)$.
\end{proposition}
\begin{remark}
The choice of topology $\mathcal{T}_t$ will allow the applications of our duality argument in non-locally compact settings, and to avoid the issue of $C_b(\Omega)^*$ being identified with the set of all regular, signed, finite and finitely additive Borel measures (\cite{dunford1958linear} Theorem IV.6.2) under the usual uniform norm topology.
\end{remark}

Let $\Mset_+(\XX)\subset\Mset(\XX)$ denote subset of positive measures. 
For any $\mu\in\Mset(\XX)$, let $L^1(\XX,\mu)$ be the set of $\mu$-integrable functions.
Also let $C_b(\XX; \R^m)$, $\Mset(\XX; \R^m)$, $L^1(\XX, \mu; \R^m)$ and so on be their respective vector valued versions.  In this paper, the typical choices of $\XX$ are $\Omega, \Lambda, \R^m$ as well as their various subspaces.

Let $\Pset$ be the set of Borel probability measures on $(\Omega,\Filt_1)$. For each $s\in[0,1]$, let $\Pset^0_s\subset\Pset$ be a subset of measures such that, for each $\P\in\Pset^0_s$, $X\in \Omega$ is an $(\FF,\P)$-semimartingale on $[s,1]$ given by
\[
X_t=X_s+A^\P_t+M^\P_t,\quad \langle X\rangle_t=\langle M^\P\rangle_t=B^\P_t, \quad \P\text{-}\as, \quad t\in[s,1],
\]
where $M^\P$ is an $(\FF,\P)$-martingale on $[s,1]$ and $(A^\P,B^\P)$ is $\FF$-adapted and $\P$-$\as$ absolutely continuous with respect to time. In particular, $\P$ is said to be have characteristics $(\alpha^\P,\beta^\P)$, which are defined by
\[
\alpha^\P_t=\frac{d A^\P_t}{dt},\quad \beta^\P_t=\frac{d B^\P_t}{dt}.
\]
Note that $(\alpha^\P,\beta^\P)$ is $\FF$-adapted and determined up to $d\P\times dt$, almost everywhere. 
In general, $(\alpha^\P,\beta^\P)$ takes values in the space $\R^{d}\times\S^d_+$. Note that $\S^d, \S^d_+$ and $\S^d_{++}$ denote the sets of symmetric matrices, positive semidefinite matrices and positive definite matrices, respectively. For any $a,b\in \S^d$, let us define $a:b := \tr(a^\intercal b)$.
Denote by $\Pset^1_t\subset \Pset^0_t$ the set of probability measures $\P$ whose characteristics $(\alpha^\P,\beta^\P)$ are $\P$-integrable on the interval $[t,1]$. In other words,
\[
\E^\P\bigg(\int_t^1 |\alpha^\P| +|\beta^\P|\, dt\bigg)< +\infty,
\]
where $|\cdot|$ denotes the $L^1$-norm.

\begin{remark}
One can define $\beta^\P$ for every each path as the quadratic variation of $X$ and it would be compatible to every semimartingale measure $\P\in\Pset^0_0$. However, this choice of $\beta^\P$ is highly pathological. Since, for each $\P$, the characteristic is determined up to $d\P\times dt$, it is usually more practical to work with versions of $\beta^\P$ that have more regularity.
\end{remark}

For each $t\in[0,1]$ and any probability measure $\rho_t$ on $\Omega_t$, define $\Pset_t(\rho_t)\subset\Pset$ to be the set of consistent with $\rho_t$ on $\Omega_t$.
\begin{gather*}
\Pset_t(\rho_t):=\{\P:\P\in\Pset,\ \P\circ X_{\cdot\wedge t}^{-1}=\rho_t\}.
\end{gather*}
Furthermore, let $\Pset^0_t(\rho_t)= \Pset_t(\rho_t) \cap \Pset^0_t$ and $\Pset^1_t(\rho_t)= \Pset_t(\rho_t) \cap \Pset^1_t$.

The notions of path derivatives and functional It\^o calculus were originally introduced in  Dupire \cite{dupire2009functional} and Cont and Fournie \cite{cont2013functional} for c\`adl\`ag paths. Here we choose a version of the definition that focuses on the space of continuous paths, as it is more suitable for the study of semimartingale optimal transport. Specifically, we use a slight variation of the definitions found in \cite{buckdahn2015pathwise,ekren2016viscosity1,ekren2016viscosity2}.

\begin{definition}[Path derivatives and functional It\^o formula]\label{defpathderiv}
For each $t\in[0,1]$, we say $\phi\in C^{1,2}_t(\Lambda)$ if $\phi\in C_b(\Lambda)$ and there exist functions $(\Dt \phi, \Dx \phi, \Dxx \phi)\in C_b(\Lambda;\R\times\R^d\times \S^d)$ such that, for any $\P\in\Pset^1_t$ and $u\in[t,1]$, the following \emph{functional It\^o formula} holds:
\begin{gather*}
\phi(u,X)-\phi(t,X)=\int_t^u \Dt \phi\, dt+ \Dx \phi \cdot dX_t + \frac{1}{2} \Dxx \phi : d\langle X\rangle_t, \quad \P\text{-\as}
\end{gather*}
The functions $\Dt \phi, \Dx \phi, \Dxx \phi$ are known as the time derivative, first order space derivative and second order space derivative of $\phi$, respectively. 
\end{definition}


\begin{remark}
The definition of $C_0^{1,2}(\Lambda)$ in Definition \ref{defpathderiv} is more restrictive than the corresponding version from \cite{ekren2016viscosity1,ekren2016viscosity2} in the following ways. We require the set $\Pset$ to contain all measures with integrable characteristics $(\alpha^\P,\beta^\P)$, as opposed to bounded characteristics. Moreover, we also require the function and its derivatives to be bounded. 
The path derivatives $\Dt \phi, \Dx \phi, \Dxx \phi$ are unique whenever they exist. We refer to \cite{buckdahn2015pathwise,ekren2016viscosity1,ekren2016viscosity2} for more discussions on the different definitions of path derivatives, as well as comparisons to the original definitions of \cite{dupire2009functional, cont2013functional}.
\end{remark}

\subsection{Problem formulation}
Now let us define the semimartingale optimal transport problem under path dependent constraints. 

Denote by $H:\Lambda \times \R^d\times \S^d \to \R\cup \{+\infty\}$ a cost function. Define $H^*:\Lambda \times \R^d\times \S^d \to \R\cup \{+\infty\}$ so that $H^*(t,\omega,\cdot)$ is the convex conjugate of $H(t,\omega,\cdot)$.
When there is no ambiguity, we will simply write $H(\alpha,\beta):=H(t,\omega,\alpha,\beta)$ and $H^*(p,q):=H^*(t,\omega,p,q)$.
We impose the following global assumption on $H$.

\begin{assumption}\label{asscostfunction01}
(i) For each $(t,\omega)\in\Lambda$, $H(t,\omega,\cdot,\cdot)$ is a lower semi-continuous, proper convex function and $\min_{\alpha,\beta} H$ is uniformly bounded.\\
(ii) If $\beta\in \S^d\setminus \S^d_+$ then $H(\cdot,\cdot,\cdot,\beta)=+\infty$.\\
(iii) The cost function $H$ is coercive in the sense that there exist constants $p>1$ and $C>0$ such that
\[
|\alpha|^p+|\beta|^p\leq C(1+H(t,\omega,\alpha,\beta)), \quad \forall\, (t,\omega)\in\Lambda.
\]
\end{assumption}

In order to characterise the constraints, let $\Nset\subseteq\Mset(\Omega)$ be a convex set of measures. 
Assume that $\Nset$ is closed under the weak-* topology, that is, the coarsest topology under which $\mu\to\int_\Omega \psi\,d\mu$ is continuous for all $\psi\in C_b(\Omega)$.
In our problem, we would like to restrict the probability measures to the set $\Nset$.

\begin{definition}
Given $t,\rho_t, H$ and $\Nset$, we define the \emph{semimartingale optimal transport problem under path dependent constraints} to be the following minimisation problem
\begin{gather*}
\inf_{\P\in\Pset^0_t(\rho_t)} \EP \int_t^1 H(\alpha^\P_s,\beta^\P_s) \,ds,\quad
\text{subject to}\quad \P \in \Nset.
\end{gather*}
The problem is said to be \emph{admissible} if $\Pset^0_t(\rho_t)\cap \Nset\neq \emptyset$ and the infimum above is finite.
\end{definition}
To handle the constraint, consider the convex function $F:C_b(\Omega)\to \R \cup\{+\infty\}$ defined by
\begin{align*}
F(\psi)
&:=\sup_{\mu\in \Nset} \int_{\Omega} \psi\,d\mu.
\end{align*}
Since $\Nset$ is closed, the convex conjugate of $F$ has the following representation: for any $\mu\in\Mset(\Omega)$,
\begin{align}\label{eqdeff01}
F^*(\mu)&
=\sup_{\psi\in C_b(\Omega)} \int_{\Omega} \psi\,d\mu - F(\psi)
=
\begin{cases}
0, & \mu\in\Nset,\\
+\infty, & \mu\notin\Nset.
\end{cases}
\end{align}

In many cases, it is possible to further restrict the choice of $\psi$ in the definition of $F^*$ to some convex subset $\Nset^*\subseteq C_b(\Omega)$. This occurs whenever the supremum in \eqref{eqdeff01} is always achieved by the elements of $\Nset^*$, so
\begin{align}\label{eqdeff02}
F^*(\mu)&=\sup_{\psi\in \Nset^*} \int_{\Omega} \psi\,d\mu - F(\psi)
=
\begin{cases}
0, & \mu\in\Nset,\\
+\infty, & \mu\notin\Nset.
\end{cases}
\end{align}
For example, if $\Nset$ is the subspace of $\Mset(\Omega)$ defined by $\Nset=\{\mu: \int_\Omega \psi'\, d\mu=0\}$ where $\psi'\in C_b(\Omega)$ is a fixed function, then we may choose $\Nset^*=\{\lambda\psi':\lambda\in\R\}$. More examples can be found in the next subsection.
In general, suitable choices of $\Nset^*$ cannot always be easily identified. However, when it is possible, the reduction of $C_b(\Omega)$ to $\Nset^*$ can greatly simplify the problem.

The formulation of $F^*$ in \eqref{eqdeff02} indicates that it is, in fact, a suitable function for the penalisation of measures outside $\Nset$.
Hence, the problem can be reformulated as the following saddle point problem:
\begin{align}
&\inf_{\P\in\Pset^0_t(\rho_t)} \EP \int_t^1H(\alpha^\P_s,\beta^\P_s) \,ds,\quad \text{ subject to } \P \in \Nset\nonumber\\
={}&\inf_{\P\in\Pset^0_t(\rho_t)} F^*(\P) + \EP \int_t^1 H(\alpha^\P_s,\beta^\P_s) \,ds\nonumber\\
={}&\inf_{\P\in\Pset^0_t(\rho_t)}\sup_{\psi\in \Nset^*}\E^\P \psi-F(\psi)   +\EP \int_t^1 H(\alpha^\P_s,\beta^\P_s) \,ds.\label{eqdeff11}
\end{align}

\subsection{Examples}\label{secexamples}
The constraint on measures in our formulation is very general and allows a wide range of problems, including many existing formulations of continuous time optimal transport problems in literature. Here are some examples.

\begin{example}[Deterministic optimal transport]\label{examplea05}
If the cost function satisfies
\[
H(\alpha,\beta)=
\begin{cases}
\frac12 |\alpha|^2, &\quad \beta=0,\\
+\infty,&\quad \beta\neq 0,
\end{cases}
\]
then we recover the classical deterministic optimal transport problem of Benamou-Brenier \cite{benamou2000computational}.
\end{example} 


\begin{example}[Semimartingale optimal transport]\label{examplea01} Consider the problem on $t\in[0,1]$ with the initial measure $\rho_0$.
Let $\bar\rho$ be a probability measure on $\R^d$ and consider the time $1$ constraint $\P\circ X_1^{-1}=\bar\rho$. Then by setting 
\[
\Nset=\{\mu\in\Mset(\Omega):\mu\circ X_1^{-1}=\bar\rho\},\  \Nset^*=\{\psi\in C_b(\Omega):\psi=\lambda\circ X_1,\ \lambda\in C_b(\R^d)\},
\]
we recover the semimartingale optimal transport problem studied in \cite{tan2013optimal}. In particular 
\begin{align*}
F(\psi)=\begin{cases}\displaystyle
\int_{\R^d} \lambda\, d\bar\rho, & \lambda\in C_b(\R^d),\  \psi=\lambda\circ X_1,\\
+\infty, & \text{otherwise},
\end{cases}
\end{align*}
and the saddle point problem is given by
\[
\inf_{\P\in\Pset^0_0(\rho_0)}\sup_{\lambda\in C_b(\R^d)} \E^\P (\lambda(X_1))-\int_{\R^d} \lambda \, d\bar\rho  +\EP \int_0^1 H(\alpha^\P_s,\beta^\P_s) \,ds.
\]

Furthermore, if the cost function is given by
$H(t,\omega_{\cdot\wedge t},\alpha,\beta)=H(t,X_t,\alpha)$ if $\beta=I_{d\times d}$, and $+\infty$ otherwise, 
then we recover the stochastic optimal control problem in \cite{mikami2006duality}.
Alternatively, if the cost function is of the form
$H(t,\omega_{\cdot\wedge t},\alpha,\beta)=H(t,X_t,\beta)$ if $\alpha=0$, and $+\infty$ otherwise, 
then we recover the martingale optimal transport problem from \cite{huesmann2019benamou}. Finally, if $d=1$ and $H(\cdot,\beta)=(\sqrt{\beta}-c)^2$ if $\alpha=0$ for some constant $c>0$, and $+\infty$ otherwise, then our formulation is equivalent to the one-dimensional case in \cite{veraguas2017martingale}.
\end{example}

\begin{example}\label{examplea02}
Further generalising the previous example, let $G\in C(\Omega;\R^m)$ be any continuous function and $\bar\rho\in\Mset(\R^m)$ be a target distribution. We would like to impose the constraint $\P \circ G^{-1}=\bar\rho$.
In this case the constraint is characterised by 
\[
\Nset=\{\mu\in\Mset(\Omega):\mu\circ G^{-1}=\bar\rho\},\  \Nset^*=\{\psi\in C_b(\Omega):\psi=\lambda\circ G,\ \lambda\in C_b(\R^{m})\},
\]
and
\begin{align*}
F(\psi)=\begin{cases}\displaystyle
\int_{\R^m} \lambda\, d\bar\rho, & \lambda\in C_b(\R^m),\  \psi=\lambda\circ G,\\
+\infty, & \text{otherwise}.
\end{cases}
\end{align*}
The saddle point problem is given by
\[
\inf_{\P\in\Pset^0_0(\rho_0)}\sup_{\lambda\in C_b(\R^m)} \E^\P (\lambda\circ G)-\int_{\R^d} \lambda \, d\bar\rho  +\EP \int_0^1 H(\alpha^\P_s,\beta^\P_s) \,ds.
\]
\end{example}

\begin{example}\label{examplea03}
Fix $G\in C_b(\Omega;\R^m)$ and consider the constraint $\E^\P G=c$. This corresponds to 
\[
\Nset=\{\mu\in\Mset(\Omega):\int_{\Omega} G \,d\mu = c\},\quad \Nset^*=\{\psi\in C_b(\Omega):\psi=\lambda\cdot G,\ \lambda\in \R^m\}.
\]
In this case,
\begin{align*}
F(\psi)=\begin{cases}\displaystyle
\lambda\cdot c , & \psi=\lambda\cdot G,\ \lambda\in \R^m,\\
+\infty, & \text{otherwise}.
\end{cases}
\end{align*}
Then the saddle point problem is given by
\[
\inf_{\P\in\Pset^0_t(\rho_t)}\sup_{\lambda\in \R^m}  \lambda \cdot (\E^\P G  - c)+\EP \int_t^1 H(\alpha^\P_s,\beta^\P_s) \,ds.
\]
\end{example} 

\begin{example}\label{examplea04}
Let $G:\Omega\to \R^m_+$ be a function that is lower semi-continuous in each component. Consider the constraint $\E^\P G\leq c$ for some $c\in\R^m_+$, where the inequality is taken element-wise. This corresponds to $\Nset=\{\mu\in\Mset_+(\Omega): \int_{\Omega} G \,d\mu \leq c\}$. One can check that this set is in fact closed under the weak-* topology\footnote{Let $\{\mu_i\}\subset \Nset$ be a sequence of measures converging to $\mu$. Let $\{G_j\}\subset C_b(\Omega;\R^m_+)$ be an increasing sequence of functions converging to $G$. Then by Fatou's lemma, $\int_\omega G\, d\mu\leq \sup_j \int_\omega G_j\, d\mu = \sup_j \lim_i \int_\omega G_j\, d\mu_i \leq \limsup_i \int_\omega G\, d\mu_i \leq c$.}. In this case
\begin{align*}
F(\psi)=\sup_{\mu\in \Nset} \int_{\Omega} \psi\,d\mu=
\inf\{\lambda\cdot c: \lambda\in\R^m_+, \ \psi\leq \lambda \cdot G\}.
\end{align*}
Using the density of $C_b(\Omega)$ in $L^1(\Omega,\P)$, the saddle point problem is given by
\begin{align*}
&\inf_{\P\in\Pset^0_t(\rho_t)}\sup_{\substack{ \psi\in C_b(\Omega)\\ \lambda\in\R^m_+,  \psi \leq \lambda \cdot G}}  \E^\P  \psi - \lambda \cdot c +\EP \int_t^1 H(\alpha^\P_s,\beta^\P_s) \,ds\\
={}&\inf_{\P\in\Pset^0_t(\rho_t)}\sup_{\substack{\lambda\in\R^m_+ }}  \lambda \cdot (\E^\P  G - c) +\EP \int_t^1 H(\alpha^\P_s,\beta^\P_s) \,ds.
\end{align*}
Via a suitable translation, the condition $G:\Omega\to \R^m_+$ can be relaxed so that $G$ is bounded from below.
\end{example} 

%

\section{Main results}

\subsection{Summary of main results}

\begin{theorem}\label{thmsummary01}
Let $\Nset\subseteq \Mset(\Omega)$ be a convex subset that is closed with respect to the weak-* topology. Define $F:C_b(\Omega)\to\R\cup\{+\infty\}$ and $F^*:\Mset(\Omega)\to\R\cup\{+\infty\}$ by
\begin{align}
F(\psi)&:=\sup_{\mu\in\Nset} \int_\Omega \psi\, d\mu,\nonumber\\
F^*(\mu)&:=\sup_{\psi\in C_b(\Omega)} \int_{\Omega} \psi\,d\mu - F(\psi)
=
\begin{cases}
0, & \mu\in\Nset,\\
+\infty, & \mu\notin\Nset.
\end{cases}\label{eqsummarya02}
\end{align}
Let $H:\Lambda\times \R^d\times\S^d \to \R\cup\{+\infty\}$ be a function satisfying Assumption \ref{asscostfunction01} and let $H^*(t,\omega,\cdot)$ denote the convex conjugates of $H(t,\omega,\cdot)$.

Given $t\in[0,1]$ and a probability measure $\rho_t$ on $\Omega_t$, recall that the semimartingale optimal transport problem with path dependent constraints refers to the following minimisation problem,
\begin{align}\label{eqsummary02}
V:=
\inf_{\P\in\Pset^0_t(\rho_t)}F^*(\P)+\EP \int_t^1 H(\alpha^\P_s,\beta^\P_s) \,ds.
\end{align}

\noindent (i) \textbf{Duality:} If the problem is admissible, then the infimum in \eqref{eqsummary02} is attained and it equals
\begin{align}\label{eqsummary03}
V=\VV:=\sup_{\psi\in C_b(\Omega)} -F(\psi) - J_{\psi,H}(t,\rho_t),
\end{align}
where $J$ is given by
\begin{align}
J_{\psi,H}(t,\rho_t)&:=\sup_{\P\in\Pset^0_t(\rho_t)} -\E^\P \psi-\EP \int_t^1 H(\alpha^\P_s,\beta^\P_s) \,ds\nonumber\\
&=\inf_{\phi\in C^{1,2}_t(\Lambda)} \int_{\Omega_t} \phi(t,\cdot)\, d\rho_t,\label{eqsummary042}\\
\text{subject to}\quad &
\phi(1,\cdot)\geq -\psi \quad \text{and}  \quad \Dt\phi + H^*\left(\Dx\phi, \frac{1}{2}\Dxx\phi\right)\leq 0.\nonumber
\end{align}
Let $\delta_{\omega_{\cdot\wedge t}}$ denote the Dirac measure with a singular mass at $\omega_{\cdot\wedge t}$. Then the function $J_{\psi,H}$ also satisfies 
\begin{align}
J_{\psi,H}(t,\rho_t)&=\int_{\Omega_t} J_{\psi,H}(t,\delta_{\omega_{\cdot\wedge t}})\, d\rho_t.
\end{align}
Moreover, if the set $C_b(\Omega)$ can be replaced by a convex subset $\Nset^*\subset C_b(\Omega)$ in \eqref{eqsummarya02}, then the same replacement can be made in \eqref{eqsummary03}.

\noindent (ii) \textbf{Optimal solution:}
Suppose that $\tilde \P$ is an optimal probability measure for the optimal transport problem and has characteristics $(\tilde \alpha, \tilde \beta)$ on $[t,1]$.
Let $(\psi^n, \phi^n)_{n\in\Z^+}$ be an optimising sequence of \eqref{eqsummary03} and \eqref{eqsummary042}. Then we have the following convergences in probability on $\Omega\times[t,1]$:
\begin{gather*}
\Dt\phi^n + H^*\left(\Dx\phi^n, \frac{1}{2}\Dxx\phi^n\right) \stackrel{d\tilde\P\times dt}{\to} 0,\quad
\phi^n+\psi^n \stackrel{d\tilde\P}{\to} 0.
\end{gather*}
Moreover, suppose that $H$ is strongly convex, i.e., there exists a constant $C>0$ such that for all $t, \omega, \alpha, \beta, \alpha', \beta'$ and any subderivative $\nabla H$, if $H(\alpha, \beta)$ is finite then
\[
H(\alpha', \beta')\geq H(\alpha, \beta) + \langle\nabla H(\alpha, \beta) , (\alpha'-\alpha, \beta'-\beta)\rangle + C (\norm{\alpha'-\alpha}^2+\norm{\beta'-\beta}^2),
\]
where $\norm{\cdot}$ is the $L^2$ norm on $\R^d$ and $\S^d$. Then the following holds on $\Omega\times[t,1]$,
\begin{gather*}
\nabla H^*\left(\Dx\phi^n, \frac{1}{2}\Dxx\phi^n\right) \stackrel{d\tilde\P\times dt}{\to} (\tilde \alpha, \tilde \beta).
\end{gather*}

\noindent (iii) \textbf{PPDE characterisation:}
Suppose that Assumptions \ref{asscostfunction01} and \ref{asscostfuncionbb} are satisfied. Then $J_{\psi,H}$ has the representation
\begin{align}
J_{\psi,H}(t,\rho_t)=\sup_{\P\in\Pset^0_t(\rho_t)} -\E^\P \psi-\EP \int_t^1 H(\alpha^\P_s,\beta^\P_s) \,ds=\int_{\Omega_t} \phi_{\psi,H}(t,\cdot)\, d\rho_t.
\end{align}
where $\phi_{\psi,H}$ is a viscosity solution of the following path dependent PDE:
\begin{gather}\label{eqmainresppde}
\phi(1,\cdot)= -\psi \quad \text{and}  \quad \Dt\phi + H^*\left(\Dx\phi, \frac{1}{2}\Dxx\phi\right)= 0.
\end{gather}
Moreover, if Assumption \ref{asscostfuncioncc} is also satisfied, then $\phi_{\psi,H}$ is the unique viscosity solution of the PPDE \eqref{eqmainresppde}.


\noindent (iv) \textbf{Dimension reduction:}
Fix $t\in[0,1]$ and $\psi\in C_b(\Omega)$. Suppose that Assumptions \ref{asscostfunction01} and \ref{asscostfuncionbb} are satisfied.  Let $\GG=\{\Gfilt_t,t\in[0,1]\}$ be a semifiltration (see Definition \ref{deflocfilt}) such that $(\psi,H)$ is adapted to $\GG$. Then the map $J(t,\cdot):\Omega \to \R$ defined by
\[
J(t,x):=J_{\psi,H}(t,\delta(\omega_{\cdot\wedge t}=x_{\cdot\wedge t}))
\]
is $\Gfilt_t$-measurable. 

Moreover, if $\phi\in C^{1,2}_0(\Lambda)$ is a classical solution of the PPDE \eqref{eqmainresppde} and $H$ is strictly convex in $(\alpha,\beta)$, then $(\alpha_t,\beta_t)= \nabla H^*(\Dx\phi(t,\cdot), \frac{1}{2}\Dxx\phi(t,\cdot))$ is also $\Gfilt_t$-measurable. 
%
\end{theorem}

\begin{remark}
By dimension reduction, we mean the following. In general, the optimal $\alpha^\P, \beta^\P$ are path dependent, $\FF$-adapted processes. For practical applications, this is not very helpful as path dependent functions are difficult to compute in general. However, in many problems, we can show that the solution in fact only depends on a few state variables which can be usually identified from the constraint and the cost function. In essence, $\Gfilt_t$ is generated by these state variables at time $t$.

For instance, consider Example \ref{examplea01} (also found in \cite{tan2013optimal}), where the constraints are on the initial density of $X_0$ and the final density of $X_1$, while the cost function at time $t$ is not path dependent and only depends on $X_t$. In this case, the optimal $\alpha^\P, \beta^\P$ at time $t$ only depend on $(t,X_t)$. So it suffices to solve a finite dimensional classical PDE, as opposed to an infinite dimensional PPDE.

Consider another example, where the constraint is on the expectation of $f(X_{t_1}, X_{t_2})$ for some fixed $t_1<t_2$ (for financial applications this would correspond to a so-called cliquet or ratchet options), then the optimal $\alpha^\P, \beta^\P$ would simply depend on $(t,X_t)$ for $t\leq t_1$ and $(X_{t_1},t,X_t)$ for $t_1< t\leq t_2$.
\end{remark}

\begin{remark}
Theorem \ref{thmsummary01} is a generalisation of many existing continuous time optimal transport results (see the examples in Subsection \ref{secexamples}). Here we will further highlight some technical differences between our technique and other works.

There are some similarities between our approach and \cite{huesmann2019benamou}, since both works make use of the Fenchel-Rockafellar duality theorem \ref{thmfenchelrockafellar}. In \cite{huesmann2019benamou} the problem is posed in Markovian settings and restricted to martingales. The duality theorem is applied first to a compact domain, then extended to an unbounded domain via verification arguments. On the other hand, our formulation is in a non-Markovian, semimartingale setting with more general constraints. The lack of local compactness leads to additional technical challenges in the formulation of the underlying topological spaces and the application of the duality result. Unlike \cite{huesmann2019benamou}, we cannot encode our semimartingale condition in the form of a Fokker-Planck equation. Instead, we identify semimartingale measures via the functional It\^o formula in the form of Lemma \ref{lemmeasureconstraint1}.

Our work also generalises the results of \cite{tan2013optimal} to non-Markovian constraints. The approach of \cite{tan2013optimal} involves proving the convexity and lower-semicontinuity of the objective function and applying the Fenchel-Moreau theorem on the measure-function duality pairing. The resulting PDE representation is derived via classical measurable selection and dynamic programming arguments. Our approach uses the Fenchel-Rockafellar duality theorem on the function-measure duality pairing and directly obtains the dual problem in terms of path dependent PDEs while bypassing dynamic programming. In fact, our duality result would also implies the dynamic programming principle without needing to invoke measurable selection arguments.
\end{remark}

To simplify our notations, the proof of the main duality result will focus on the case $t=0$. In other words, this corresponds to the class of optimal transport problem starting at time 0 with initial distribution $\rho_0$. The general case can be dealt with using similar arguments.

\subsection{Characterising suitable measures}
The function $F^*$ was used to penalise measures outside of $\Nset$, a constraint of the problem. In this subsection, we aim to find a suitable function that penalises measures outside of $\Pset^1_0(\rho_0)$, the set of semimartingale measures.

A key step in our argument is to extend measures in $\Mset(\Omega)$ to measures on the stopped paths $\Mset(\Lambda)$, then to utilise the Fenchel-Rockafellar duality theorem \ref{thmfenchelrockafellar} (see, e.g., \cite{rockafellar1966extension}) on that space. The following lemma provides a condition for the identification of measures in $\Pset^1_0(\rho_0)$ as well as a suitable corresponding measure in $\Mset(\Lambda)$.
\begin{lemma}\label{lemmeasureconstraint1}
Let $\rho_0$ be a probability measure on $\Omega_0$. Suppose that $\mu\in\Mset_+(\Omega)$ and induces the measure $\hat\mu\in\Mset_+(\Lambda)$ via
\[
\hat\mu(E) = \int_\Omega \int_0^1 \I((t,\omega_{\cdot\wedge t})\in E) \, dt d\mu, \quad \forall\,E\subset\Lambda.
\]
Let $\nu\in\Mset_+(\Lambda)$ and $(\alpha,\beta)\in L^1(\Lambda,\nu;\R^d\times\S^d)$.

The equality
\begin{gather}\nonumber
\int_{\Omega} \phi(1,\cdot) \,d\mu- \int_{\Omega_0} \phi(0,\cdot) \,d\rho_0 = \int_{\Lambda} \Dt \phi + \alpha \cdot \Dx \phi + \frac{1}{2}\beta : \Dxx\phi \,d\nu
\end{gather}
holds for all $\phi\in C^{1,2}_0(\Lambda)$ if and only if $\hat\mu = \nu$ and $\mu\in\Pset^1_0(\rho_0)$ has characteristics $(\alpha,\beta)$.


\end{lemma}
\begin{proof}
See Appendix. 
\end{proof}

\subsection{Duality}

In this subsection, we will prove Theorem 3.1 (i), the main duality result of the paper. In fact, we will prove a slightly stronger statement in Theorem \ref{thmduality1}, which allows $F$ to be an arbitrary convex function, rather than one defined using the set $\Nset$.

Recall that, at time $t=0$, the saddle point problem is given by
\[
\inf_{\P\in\Pset^0_0(\rho_0)}\sup_{\psi\in \Nset^*} \E^\P \psi- F(\psi)   +\EP \int_0^1 H(s,X_{\cdot\wedge s},\alpha^\P_s,\beta^\P_s) \,ds.
\]
The key step is to encode the condition $\P\in\Pset^0_0(\rho_0)$ using Lemma \ref{lemmeasureconstraint1} and reformulate the problem with respect to the measures $\mu$ and $\nu$,
\[
\inf_{\P\in\Pset^0_0(\rho_0)}\sup_{\psi\in \Nset^*}=\inf_{\substack{\mu\in\Mset_+(\Omega), \nu\in\Mset_+(\Lambda)\\ (\alpha,\beta)\in L^1(\Lambda,\nu;\R^d\times\S^d)}}\sup_{\substack{\phi\in C^{1,2}_0(\Lambda), \psi\in\Nset^*}} .
\]
Then duality (swapping the infimum and the supremum) is established via the Fenchel-Rockafellar duality theorem \ref{thmfenchelrockafellar}. 
We will state the version of the theorem found in \cite{rockafellar1966extension}. Variants of the theorem can also be found in, e.g., the beginning of \cite{brezis1983analyse}, or Theorem 1.9 in \cite{villani2003topics}.
\begin{theorem}[Fenchel-Rockafellar]\label{thmfenchelrockafellar}
Let $E$ be a locally convex Hausdorff topological vector space over $\R$ with dual $E^*$. Let $f:E\to \R\cup\{+\infty\}$ be a proper convex function, $g:E\to \R\cup\{-\infty\}$ be a proper concave function, and $f^*$, $g_*$ be their respective conjugates. If either $f$ or $g$ is continuous at some point where both functions are finite, then
\[
\inf_{x\in E} f(x)-g(x) = \max_{y\in E^*} g_*(y)-f^*(y).
\]
\end{theorem}
Roughly speaking, the duality pairings used in the current context are of the form $(\mu,\nu,\bar \nu, \tilde\nu)$ and $(\phi_1+\psi,\Dt\phi,\Dx\phi,\Dxx\phi)$, where $d\bar\nu=\alpha d\nu$ and $d\tilde\nu=\beta d\nu$.
The dual formulation will be characterised using path dependent PDEs (PPDEs).

\begin{theorem}\label{thmduality1}
Let $F:C_b(\Omega)\to\R$ be a convex function and $\Nset^*\subseteq C_b(\Omega)$ be a convex set such that $\inf_{\psi\in\Nset^*}F(\psi)<+\infty$.
Define the function $F^*:\Mset(\Omega)\to\R$ by\footnote{In this set up, unless $N^*=C_b(\Omega)$, $F^*$ is not necessarily the convex conjugate of $F$. Instead, $F^*$ would be the convex conjugate of $F\I(\cdot\in\Nset^*)$.}
\begin{align*}
F^*(\mu)&:=\sup_{\psi\in \Nset^*} \int_{\Omega} \psi\,d\mu - F(\psi).
\end{align*}
Let $H:\Lambda \times U \to \R\cup \{+\infty\}$ satisfy Assumption \ref{asscostfunction01} and $H^*(t,\omega,\cdot)$ be the convex conjugate of $H(t,\omega,\cdot)$.
Define
\begin{align}
V&:=\inf_{\P\in\Pset^0_0(\rho_0)}F^*(\P)+\EP \int_0^1 H(\alpha^\P_t, \beta^\P_t) \,dt,\label{eqthmduality01}\\
\VV&:=\sup_{\psi\in \Nset^*, \phi\in C^{1,2}_0(\Lambda)} -F(\psi) - \int_{\Omega_0} \phi(0,\cdot)\, d\rho_0, \label{eqthmduality02}\\
\text{subject to}\quad &\phi(1,\cdot)\geq -\psi \quad \text{and}  \quad \Dt\phi + H^*\left(\Dx\phi, \frac{1}{2}\Dxx\phi\right)\leq 0. \label{eqthmduality03}
\end{align}
Then $V=\VV$. Moreover, if $V$ is finite, then the infimum in \eqref{eqthmduality01} is attained. 
\end{theorem}

\begin{proof} For convenience, let us use $\phi_t$ instead of $\phi(t,\cdot)$. Recall that $\hat \P \in \Mset(\Lambda)$ is induced by $\P$ via $d\hat\P=dt\times d\P$.
Under Assumption \ref{asscostfunction01}, it suffices to only consider the set of probability measures $\P\in\Pset^1_0(\rho_0)$ in \eqref{eqthmduality01}.

The direction $V\geq \VV$ can be easily shown as follows.
Applying Fubini's theorem and Lemma \ref{lemmeasureconstraint1}, we have
\begin{align*}
&V=\inf_{\P\in\Pset^1_0(\rho_0)} F^*(\P)+ \int_\Lambda H(\alpha^\P,\beta^\P) \,d\hat\P \\
&=\inf_{\substack{\mu\in\Mset_+(\Omega), \nu\in\Mset_+(\Lambda)\\ (\alpha,\beta)\in L^1(\Lambda,\nu;\R^d\times\S^d)}}\sup_{\substack{\phi\in C^{1,2}_0(\Lambda), \psi\in\Nset^*}} 
 -F(\psi) - \int_{\Omega_0} \phi_0\, d \rho_0 +\int_{\Omega} (\psi+\phi_1) \,d\mu\\
&\qquad\qquad-\int_\Lambda (\Dt\phi + \alpha \cdot \Dx\phi + \frac{1}{2}\beta : \Dxx\phi)\,d\nu+\int_\Lambda H(\alpha,\beta) \,d\nu\\
&\geq \sup_{\substack{\phi\in C^{1,2}_0(\Lambda), \psi\in\Nset^*}}\inf_{\substack{\mu\in\Mset_+(\Omega), \nu\in\Mset_+(\Lambda)\\ (\alpha,\beta)\in L^1(\Lambda,\nu;\R^d\times\S^d)}}
 -F(\psi) - \int_{\Omega_0} \phi_0\, d \rho_0 +\int_{\Omega} (\psi+\phi_1) \,d\mu\\
&\qquad\qquad-\int_\Lambda (\Dt\phi + \alpha \cdot \Dx\phi + \frac{1}{2}\beta : \Dxx\phi)\,d\nu+\int_\Lambda H(\alpha,\beta) \,d\nu\\
&\geq \sup_{\substack{\phi\in C^{1,2}_0(\Lambda), \psi\in\Nset^*}} -F(\psi) - \int_{\Omega_0} \phi_0\, d \rho_0 + \inf_{\mu\in\Mset_+(\Omega)}\int_{\Omega} (\psi+\phi_1) \,d\mu\\
&\qquad\qquad-\sup_{\nu\in\Mset_+(\Lambda)}\int_\Lambda \bigg(\sup_{(\alpha,\beta)\in\R^d\times \S^d} \Dt\phi + \alpha \cdot \Dx\phi + \frac{1}{2}\beta : \Dxx\phi- H(\alpha,\beta)\bigg)\,d\nu\\
&=\sup_{\substack{\phi\in C^{1,2}_0(\Lambda), \psi\in\Nset^*}} -F(\psi) - \int_{\Omega_0} \phi_0\, d \rho_0 + \inf_{\mu\in\Mset_+(\Omega)}\int_{\Omega} (\psi+\phi_1) \,d\mu\\
&\qquad\qquad-\sup_{\nu\in\Mset_+(\Lambda)}\int_\Lambda \bigg(\Dt\phi + H^*\bigg(\Dx\phi,\frac{1}{2}\Dxx\phi\bigg)\bigg)\,d\nu\\
&=\sup_{\substack{\phi\in C^{1,2}_0(\Lambda), \psi\in\Nset^*}} -F(\psi) - \int_{\Omega_0}\phi_0\,d\rho_0 , \\
&\qquad\qquad\qquad\qquad \text{s.t. $\phi_1\geq -\psi,\ \Dt\phi+ H^*(\Dx\phi, \frac{1}{2}\Dxx\phi)\leq 0$.}\\
&=\VV
\end{align*}

Now let us focus on the opposite direction, $V\leq \VV$, which is significantly more difficult. 
The main technical issue is that, in order to use the Fenchel-Rockafellar duality theorem \ref{thmfenchelrockafellar}, we require the effective domain of a particular convex function to have a non-empty interior. This is achieved by introducing two additional slack terms of $\epsilon$, in both $H$ and $\phi_1+\psi$. This is similar to the technique in \cite{villani2003topics}, Section 1.3, for the proof of Kantorovich duality in Proposition 1.22.

By the conditions on $F$ and $\Nset^*$, there exists some $\psi'\in \Nset^*$ such that $F(\psi')<+\infty$.
Throughout the remainder of the proof, let $\epsilon>0$ be a fixed constant such that $\epsilon > \max(|\psi|_{\infty},  |H^*(0,0)|)=\max(|\psi|_{\infty},  |\min H|)$.

Then we have
\begin{align*}
V+2\epsilon&=\inf_{\P\in\Pset^1_0(\rho_0)} F^*(\P) +  \int_\Omega \epsilon\, d\P+ \int_\Lambda (H(\alpha^\P, \beta^\P)+\epsilon) \,d\hat\P \\
&=\inf_{\P\in\Pset^1_0(\rho_0)} \sup_{\psi\in \Nset^*} - F(\psi)  + \int_\Omega (\psi+\epsilon)\, d\P + \int_\Lambda (H(\alpha^\P, \beta^\P)+\epsilon) \,d\hat\P.
\end{align*}
Applying Fubini's theorem and Lemma \ref{lemmeasureconstraint1}, we have
\begin{align*}
V+2\epsilon&=\inf_{\substack{\mu\in\Mset_+(\Omega), \nu\in\Mset_+(\Lambda)\\ (\alpha,\beta)\in L^1(\Lambda,\nu;\R^d\times\S^d)}}\sup_{\substack{\phi\in C^{1,2}_0(\Lambda), \psi\in\Nset^*}} 
 -F(\psi)  +\int_{\Omega} (\psi+\phi_1+\epsilon) \,d\mu\\
&\qquad\qquad- \int_{\Omega_0} \phi_0\, d \rho_0  -\int_\Lambda (\Dt\phi + \alpha \cdot \Dx\phi + \frac{1}{2}\beta : \Dxx\phi-H(\alpha,\beta)-\epsilon) \,d\nu \\
&=\inf_{\substack{\mu\in\Mset_+(\Omega), \nu\in\Mset_+(\Lambda)\\ (\alpha,\beta)\in L^1(\Lambda,\nu;\R^d\times\S^d)}}
\sup_{\substack{\phi\in C^{1,2}_0(\Lambda), \psi\in\Nset^* \\ \varphi\in C_b(\Omega), \varphi= -(\phi_1+\psi)}} 
-F(\psi) +\int_{\Omega} (\epsilon-\varphi) \,d\mu \\
&\qquad\qquad- \int_{\Omega_0} \phi_0\, d \rho_0-\int_\Lambda (\Dt\phi+ \alpha \cdot \Dx\phi + \frac{1}{2}\beta : \Dxx\phi-H(\alpha,\beta)-\epsilon) \,d\nu.
\end{align*}

For the next part of the proof, introduce the measures $(\bar\nu, \tilde\nu) \in \Mset(\Lambda;\R^d\times \S^d)$ via
$d\bar\nu=\alpha d\nu$ and $d\tilde\nu=\beta d\nu$,
so that we can write
\begin{align*}
V+2\epsilon
&=\inf_{(\xi,\rho,\mu,\nu,\bar\nu,\tilde\nu)\in\AA}\sup_{(\psi,\bar\varphi,\varphi,p,q,r)\in\BB}
-F(\psi) +\int_\Lambda \left(H\left(\frac{d\bar\nu}{d\nu},\frac{d\tilde\nu}{d\nu}\right)+\epsilon\right) d\nu \\
&\qquad\qquad\qquad+\int_{\Omega} \epsilon \,d\mu - \langle (\xi,\rho,\mu,\nu,\bar\nu,\tilde\nu),(\psi,\bar\varphi,\varphi,p,q,r) \rangle,
\end{align*}
where
\begin{align*}
\AA&:=\{(\xi,\rho,\mu,\nu,\bar\nu,\tilde\nu)\in \Mset({\Omega})\times \Mset({\Omega_0})\times\Mset({\Omega})\times\Mset(\Lambda;\R\times\R^d\times \S^d):\\
&\qquad\qquad \xi=0,\ \rho=\rho_0,\ \mu\geq 0,\ \nu \geq 0,\ (\bar\nu,\tilde\nu)\ll\nu\},\\
\BB&:=\{(\psi,\bar\varphi,\varphi,p,q,r)\in C_b(\Omega)\times C_b({\Omega_0})\times C_b({\Omega})\times C_b(\Lambda;\R\times\R^d\times \S^d):\\
&\qquad\qquad \psi\in\Nset^*,\ \exists\, \phi\in C^{1,2}_0(\Lambda) \ \text{ s.t. }\ \bar\varphi=\phi_0,\ \varphi= -(\phi_1+\psi),\ (p,q,r)=(\Dt,\Dx,\frac{1}{2}\Dxx)\phi\},
\end{align*}
and the inner product is defined by
\begin{align*}
\langle (\xi,\rho,\mu,\nu,\bar\nu,\tilde\nu),(\psi,\bar\varphi,\varphi,p,q,r) \rangle 
:= \int_{\Omega} (\psi\,d\xi+\varphi\,d\mu)+ \int_{\Omega_0}  \bar\varphi \,d\rho + \int_\Lambda (p\,d\nu +  q\cdot d\bar\nu +  r : d\tilde\nu).
\end{align*}
Note that both $\AA$ and $\BB$ are convex sets. 

Next, define the convex function 
$a:C_b(\Omega)\times C_b({\Omega_0})\times C_b({\Omega})\times C_b(\Lambda;\R\times\R^d\times \S^d)\to\R\cup\{+\infty\}$ and its convex conjugate 
$a^*: \Mset({\Omega})\times \Mset({\Omega_0})\times\Mset({\Omega})\times\Mset(\Lambda;\R\times\R^d\times \S^d)\to\R\cup\{+\infty\}$
according to Lemma \ref{lemcostfunciton01}. They are given by the following expressions:
\begin{align*}
a(\psi,\bar\varphi,\varphi,p,q,r) &:= \begin{cases} \displaystyle
\int_{\Omega_0} \bar\varphi\, d\rho_0, & \text{if $\varphi\leq \epsilon$ and $p+H^*(q, r)\leq \epsilon$,} \\
+\infty, & \text{otherwise,}
\end{cases}\\
a^*(\xi,\rho,\mu,\nu,\bar\nu,\tilde\nu)&:= \begin{cases} \displaystyle
\int_{\Omega} \epsilon \,d\mu+\int_\Lambda \left(H\left(\frac{d\bar\nu}{d\nu},\frac{d\tilde\nu}{d\nu}\right)+\epsilon\right) d\nu,  & \text{if $(\xi,\rho,\mu,\nu,\bar\nu,\tilde\nu)\in\AA$}, \\
+\infty,  & \text{otherwise.}
\end{cases}
\end{align*}
Furthermore define the concave function $b:C_b(\Omega)\times C_b({\Omega_0})\times C_b({\Omega})\times C_b(\Lambda;\R\times\R^d\times \S^d)\to\R\cup\{-\infty\}$ and its concave conjugate 
$b_*: \Mset({\Omega})\times \Mset({\Omega_0})\times\Mset({\Omega})\times\Mset(\Lambda;\R\times\R^d\times \S^d)\to\R\cup\{-\infty\}$ in the following way
\begin{align*}
b(\psi,\bar\varphi,\varphi,p,q,r)&:= \begin{cases} \displaystyle
-F(\psi), & \text{if $(\psi,\bar\varphi,\varphi,p,q,r)\in\BB$}, \\
-\infty, & \text{otherwise.}
\end{cases}\\
b_*(\xi,\rho,\mu,\nu,\bar\nu,\tilde\nu)&:= \inf_{(\psi,\bar\varphi,\varphi,p,q,r)\in \BB}\langle (\xi,\rho,\mu,\nu,\bar\nu,\tilde\nu),(\psi,\bar\varphi,\varphi,p,q,r) \rangle +F(\psi) .
\end{align*}
Note that we do not need to compute $b_*$ explicitly.
Hence $V$ can be written as
\begin{align*}
V+2\epsilon&= \inf_{(\xi,\rho,\mu,\nu,\bar\nu,\tilde\nu)}\sup_{(\psi,\bar\varphi,\varphi,p,q,r)}  a^*(\xi,\rho,\mu,\nu,\bar\nu,\tilde\nu) + b(\psi,\bar\varphi,\varphi,p,q,r)\\
&\qquad\qquad\qquad\qquad\qquad\qquad - \langle (\xi,\rho,\mu,\nu,\bar\nu,\tilde\nu),(\psi,\bar\varphi,\varphi,p,q,r) \rangle,\\
&=\inf_{(\xi,\rho,\mu,\nu,\bar\nu,\tilde\nu)}  a^*(\xi,\rho,\mu,\nu,\bar\nu,\tilde\nu) - b_*(\xi,\rho,\mu,\nu,\bar\nu,\tilde\nu).
\end{align*}
In order to apply the Fenchel-Rockafellar duality theorem \ref{thmfenchelrockafellar}, we require $\operatorname{cont}(a)\cap \operatorname{dom}(b) \neq \emptyset$. Recall that $F(\psi')<+\infty $, $-\psi'<\epsilon$ and $H^*(0,0)<\epsilon$. The required condition is fulfilled at $(\psi,\bar\varphi,\varphi,p,q,r)=(\psi',0,-\psi',0,0,0)$. 
The duality theorem implies that
\begin{align*}
V+2\epsilon&=\sup_{(\psi,\bar\varphi,\varphi,p,q,r)} b(\psi,\bar\varphi,\varphi,p,q,r)- a(\psi,\bar\varphi,\varphi,p,q,r)\\
&=\sup_{(\psi,\bar\varphi,\varphi,p,q,r)\in \BB}   -F(\psi) -  \int_{\Omega_0} \bar\varphi\, d\rho_0, \quad\text{s.t. $\varphi\leq \epsilon$, $p+H^*(q, r)\leq \epsilon$},\\
&=\sup_{\substack{\phi\in C^{1,2}_0(\Lambda), \psi\in\Nset^*}} -F(\psi) - \int_{\Omega_0}\phi_0\,d\rho_0 , \\
&\qquad\qquad\qquad \text{s.t. $\phi_1+\psi\geq -\epsilon$, $\Dt\phi+ H^*(\Dx\phi, \frac{1}{2}\Dxx\phi)\leq \epsilon$,}
\end{align*}
Translating $\phi$ by $2\epsilon$ yields
\begin{align}
V&= \sup_{\substack{\phi\in C^{1,2}_0(\Lambda), \psi\in\Nset^*}} -F(\psi) - \int_{\Omega_0}\phi_0\,d\rho_0 , \nonumber\\
&\qquad\qquad\qquad \text{s.t. $\phi_1+\psi\geq \epsilon$ and $\Dt\phi+ H^*(\Dx\phi, \frac{1}{2}\Dxx\phi)\leq \epsilon$.} \label{eqphicondition1}
\end{align}
In order to eliminate $\epsilon$ in \eqref{eqphicondition1}, we will use the fact that one can arbitrarily modified the time derivative of a path dependent function without altering the space derivatives (see Remark \ref{rempathdepfuncalter}).
For each $\phi\in C^{1,2}_0(\Lambda)$ satisfying \eqref{eqphicondition1}, we can construct $\phi'\in C^{1,2}_0(\Lambda)$ via
\[
\phi'(t,\cdot)=\phi(t,\cdot)-\int_0^t \left(\Dt\phi(s,\cdot)+ H^*(\Dx\phi(s,\cdot), \frac{1}{2}\Dxx\phi(s,\cdot))\right)^+ ds.
\]
It is straightforward to check that
\begin{gather*}
\phi'_0=\phi_0,\quad \phi'_1\geq\phi_1-\epsilon,\quad (\Dx,\Dxx)\phi'=(\Dx,\Dxx)\phi, \\
\Dt\phi'=\Dt\phi-\left(\Dt\phi+ H^*(\Dx\phi, \frac{1}{2}\Dxx\phi)\right)^+\leq -H^*(\Dx\phi', \frac{1}{2}\Dxx\phi').
\end{gather*}
Therefore 
\begin{align*}
V&\leq \sup_{\substack{\phi\in C^{1,2}_0(\Lambda), \psi\in\Nset^*}} -F(\psi) - \int_{\Omega_0}\phi_0\,d\rho_0 , \\
&\qquad\qquad\qquad \text{s.t. $\phi_1+\psi\geq 0$ and $\Dt\phi+ H^*(\Dx\phi, \frac{1}{2}\Dxx\phi)\leq 0$}\\
&=\VV.
\end{align*}
Thus we may conclude $V=\VV$. The fact that the infimum in \eqref{eqthmduality01} is attained if $V<+\infty$ is a direct consequence of the Fenchel-Rockafellar duality theorem \ref{thmfenchelrockafellar}. This completes the proof of Theorem \ref{thmduality1}.
\end{proof}

\begin{corollary}\label{corduality15} The analogous result for time $t>0$ is given below.
\begin{gather}
\inf_{\P\in\Pset^0_t(\rho_t)} F^*(\P)+\EP \int_t^1 H(s,X_{\cdot\wedge s},v^\P_s) \,ds
=\sup_{\psi\in\Nset^*,\phi\in C^{1,2}_t(\Lambda)} -F(\psi) - \int_{\Omega_t}\phi(t,\cdot)\,d\rho_t,\nonumber \\
\text{subject to}\quad \phi(1,\cdot)\geq -\psi \quad \text{and}  \quad \Dt\phi + H^*\left(\Dx\phi, \frac{1}{2}\Dxx\phi\right)\leq 0.\label{eqcordualityppde}
\end{gather}
\end{corollary}
\begin{remark}\label{rempathdepfuncalter}
By adding $\int_0^\cdot f \, dt$ to any function in $C^{1,2}_0(\Lambda)$, it is possible to change the time derivative by an arbitrary function $f$ without altering the space derivatives.
This is a useful yet peculiar property of path dependent derivatives and it is somewhat counter-intuitive when compared to conventional differentiable functions. By applying this idea and increasing $\Dt\phi$ appropriately, we can replace the inequality in the PPDE by an equality. However, the inequality in the terminal condition remains.
\end{remark}

\begin{remark}\label{remrobusthedging}
Theorem \ref{thmduality1} provides a slightly more general set up than Theorem \ref{thmsummary01} (i), due to the additional flexibility in $F$. Beyond the examples in Subsection \ref{secexamples}, this allows for an even wider range of applications, e.g., the inclusion of a terminal objective function in the primal problem. A noteworthy example is the application to the robust hedging of path dependent options in continuous time settings (see, e.g., \cite{dolinsky2014martingale, hou2018robust}). As shown below, Theorem \ref{thmduality1} immediately implies the so-called \emph{robust pricing hedging duality}.

Consider the problem of robust hedging a European payoff $f\in C_b(\Omega)$ with respect to a set of models which are represented by the martingale measures whose diffusion characteristic is bounded lines in some compact and convex set $D\subseteq \S^d_+$. Suppose there are additional model constraints in the form of European payoffs $g\in C_b(\Omega,\R^m)$ with the known price of 0. Let us set $F=0$, $\Nset^*=\{\lambda\cdot g - f: \lambda \in\R^m\}$, and $H(\alpha,\beta)=0$ if $\alpha=0, \beta\in D$ or $+\infty$ otherwise. Then the robust model price is equal to the primal problem,
\[
\sup_{\substack{\P\in \Pset^0_0, \EP g= 0\\ \alpha^\P=0, \beta^\P\in D}} \EP f = V =
\sup_{\P\in \Pset^0_0} \inf_{\lambda \in \R^m} -\lambda \cdot \EP g + \EP f - \EP \int_0^1 H(\alpha^{\P},\beta^{\P})\, dt.
\]
The dual formulation is given by
\begin{gather} \label{eqthmduality033}
\mathcal{V} = \inf_{\substack{\lambda \in \R^m\\ \phi\in C^{1,2}_0(\Lambda)}} \phi(0,X_0),\quad \text{ s.t. }
\quad \phi(1,\cdot)\geq f-\lambda\cdot g, \ \Dt\phi + \sup_{\beta \in D} \frac{1}{2}\beta : \Dxx\phi\leq 0.
\end{gather}
By the functional It\^o formula, for each $\phi$ satisfying \eqref{eqthmduality033}, the following inequality holds $\P$-$\as$ for every candidate model $\P$,
\begin{align*}
f-\lambda\cdot g -\phi(0,X_0) &\leq \int_0^1 (\Dt \phi+ \,\frac{1}{2} \beta^\P:\Dxx \phi) dt+ \Dx \phi \cdot dX_t\\
&\leq \int_0^1 \Dx \phi \cdot dX_t.
\end{align*}
Hence the trading strategy with a starting cost of $\phi(0,X_0)$, holding a dynamic portfolio $\Dx \phi$ in $X$ and a static position $\lambda$ in $g$, is a superhedging strategy for $f$. Since this is true for all $\phi$ satisfying \eqref{eqthmduality033}, the dual value $\mathcal{V}$ must be an upper bound for the superhedging price, which is easily shown to be at least the robust model price. By Theorem \ref{thmduality1}, we must have equality between the robust model price and the superhedging price.
\end{remark}

To complete the proof of Theorem \ref{thmsummary01} (i), we need to express the solution of the path dependent optimal transport problem in terms of the function $J$, which reaffirms the link between optimal control problems and PPDEs.
\begin{proposition}\label{corduality2}
Fix $\psi\in C_b(\Omega)$.
Define $J$ by 
\begin{align}\label{eqjj01}
J_{\psi,H}(t,\rho_t)&:=\sup_{\P\in\Pset^0_t(\rho_t)} -\E^\P \psi-\EP \int_t^1 H(\alpha^\P_s,\beta^\P_s) \,ds.
\end{align}
(i) Then $J_{\psi,H}$ can be written as
\begin{gather*}
J_{\psi,H}(t,\rho_t)=\inf_{\phi\in C^{1,2}_t(\Lambda)} \int_{\Omega_t} \phi(t,\cdot)\, d\rho_t,\\
\text{subject to}\quad 
\phi(1,\cdot)\geq -\psi \quad \text{and}  \quad \Dt\phi + H^*\left(\Dx\phi, \frac{1}{2}\Dxx\phi\right)\leq 0.
\end{gather*}
(ii) The function $J_{\psi,H}$ also satisfies
\begin{gather*}
J_{\psi,H}(t,\rho_t)=\int_{\Omega_t} J_{\psi,H}(t,\delta_{\omega_{\cdot\wedge t}})\, d\rho_t.
\end{gather*}
\end{proposition}
\begin{proof}
(i) This is an immediate consequence of Theorem \ref{thmduality1} after setting $F=0$ and $\Nset^*=\{\psi\}$.


\noindent (ii) Define the set
\[
\Phi:=\bigg\{\phi\in C^{1,2}_t(\Lambda): \phi(1,\cdot)\geq -\psi,\ \Dt\phi + H^*\left(\Dx\phi, \frac{1}{2}\Dxx\phi\right)\leq 0\bigg\}.
\]
From part (i), we immediately have
\[
J_{\psi,H}(t,\rho_t)\geq \int_{\Omega_t} \inf_{\phi\in \Phi}  \phi(t,\omega_{\cdot\wedge t})\, d\rho_t = \int_{\Omega_t} J_{\psi,H}(t,\delta_{\omega_{\cdot\wedge t}})\, d\rho_t.
\]
For the other direction, for any $\P\in\Pset^0_t(\rho_t)$, let us disintegrate $\P$ with respect to the map $X_{\cdot\wedge t}$. Hence there exists a $\rho_t$-\as unique map $\omega_{\cdot\wedge t} \to \P_{\omega_{\cdot\wedge t}} \in \Pset_t(\delta_{\omega_{\cdot\wedge t}})$ such that, for all $E\subset \Omega$,
\[
\P(E)=\int_{\Omega_t}  \P_{\omega_{\cdot\wedge t}}(E)\, d\rho_t.
\]
Recall that $X$ is an $(\FF,\P)$-semimartingale with characteristics $(\alpha^\P,\beta^\P)$. Then it follows that, $\rho_t$-\as, $X$ is an $(\FF,\P_{\omega_{\cdot\wedge t}})$-semimartingale also with characteristics $(\alpha^\P,\beta^\P)$. Hence
\begin{align*}
\E^\P \bigg({-\psi}-\int_t^1 & H(\alpha^\P_s,\beta^\P_s)  \,ds\bigg)= \int_{\Omega_t} \E^{\P_{\omega_{\cdot\wedge t}}} \bigg({-\psi}-\int_t^1 H(\alpha^\P_s,\beta^\P_s) \,ds\bigg) \, d\rho_t\\
&\leq \int_{\Omega_t} \sup_{\P'\in\Pset^0_t(\delta_{\omega_{\cdot\wedge t}})} \E^{\P'} \bigg({-\psi}-\int_t^1 H(\alpha^{\P'}_s,\beta^{\P'}_s) \,ds\bigg) \, d\rho_t\\
&=\int_{\Omega_t} J_{\psi,H}(t,\delta_{\omega_{\cdot\wedge t}}) \, d\rho_t.
\end{align*}
Since this holds for any $\P\in\Pset^0_t(\rho_t)$, by definition of $J_{\psi,H}$, we obtain the required result.
\end{proof}

\begin{remark}
The argument used in Proposition \ref{corduality2} can be extended to obtain the well-known dynamic programming principle for $J_t=J_{\psi,H}(t,\delta_{\omega_{\cdot\wedge t}})$. To briefly outline the argument, for any stopping time $\tau\geq t$, we have the following chain of inequalities,
\[
J_t \leq \sup_\P \E^\P \bigg(J_\tau - \int_t^\tau H(\alpha^\P,\beta^\P)\, dt\bigg) = \sup_\P \E^\P \bigg(\inf_\phi \phi_\tau - \int_t^\tau H(\alpha^\P,\beta^\P)\, dt \bigg) \leq \inf_\phi \phi_t,
\]
where the first inequality follows from the disintegration argument of Proposition \ref{corduality2} (ii), the middle equality follows from the duality result of Proposition \ref{corduality2} (i), and the final inequality follows from the functional It\^o formula and the PPDE.
Applying the duality result again implies there must be equality throughout.
\end{remark}

The combination of Corollary \ref{corduality15} and Proposition \ref{corduality2} implies our main duality result, Theorem \ref{thmsummary01} (i).


\subsection{Optimal probability measure}
If the optimum of the dual problem is attained, we can characterise the optimal probability measure using Theorem \ref{thmsummary01} (ii), which is restated as Proposition \ref{propduality3} below. 
\begin{proposition}\label{propduality3}
Let $\tilde \P$ be an optimal probability measure for the optimal transport problem, with characteristics $(\tilde \alpha, \tilde \beta)$ on $[t,1]$.
Let $(\psi^n, \phi^n)_{n\in\Z^+}$ be an optimising sequence of \eqref{eqsummary03} and \eqref{eqsummary042}. Then we have the following convergences in probability on $\Omega\times[t,1]$:
\begin{gather}
\Dt\phi^n + H^*\left(\Dx\phi^n, \frac{1}{2}\Dxx\phi^n\right) \stackrel{d\tilde\P\times dt}{\to} 0,\label{eqoptidual11}\quad
\phi^n+\psi^n \stackrel{d\tilde\P}{\to} 0.
\end{gather}
Moreover, suppose that $H$ is strongly convex, i.e., there exists a constant $C>0$ such that for all $t, \omega, \alpha, \beta, \alpha', \beta'$ and any subderivative $\nabla H$, if $H(\alpha, \beta)$ is finite then
\[
H(\alpha', \beta')\geq H(\alpha, \beta) + \langle\nabla H(\alpha, \beta) , (\alpha'-\alpha, \beta'-\beta)\rangle + C (\norm{\alpha'-\alpha}^2+\norm{\beta'-\beta}^2),
\]
where $\norm{\cdot}$ is the $L^2$ norm on $\R^d$ and $\S^d$. Then the following holds on $\Omega\times[t,1]$,
\begin{gather*}
\nabla H^*\left(\Dx\phi^n, \frac{1}{2}\Dxx\phi^n\right) \stackrel{d\tilde\P\times dt}{\to} (\tilde \alpha, \tilde \beta).
\end{gather*}
\end{proposition}

\begin{proof}
For each $\epsilon>0$, we have the following inequality for all large enough $n$
\[
F^*(\tilde\P)+\E^{\tilde \P} \int_t^1 H(\tilde \alpha_s,\tilde \beta_s) \,ds \leq -F(\psi^n)-\int_{\Omega_t} \phi^n(t,\cdot)\, d\rho_t + \epsilon.
\]
Applying the functional It\^o formula and rearranging, this yields
\begin{align}
&\E^{\tilde \P} \bigg(\int_t^1 H(\tilde \alpha_s,\tilde \beta_s) + H^*\left(\Dx\phi^n_s, \frac{1}{2}\Dxx\phi^n_s\right) - \tilde\alpha_s\cdot \Dx\phi^n_s - \frac{1}{2}\tilde \beta_s :\Dxx\phi^n_s \,ds \bigg)
\label{eqoptimumcondition1} \\
&\qquad\qquad+\E^{\tilde \P} \bigg(\int_t^1 -\Dt\phi^n_s - H^*\left(\Dx\phi^n_s, \frac{1}{2}\Dxx\phi^n_s\right) ds \bigg)\nonumber \\
&\qquad\qquad+\E^{\tilde \P} (\phi^n_1+\psi^n)+(F^*(\tilde\P)+F(\psi^n)-\E^{\tilde \P} \psi^n) \leq \epsilon.\nonumber
\end{align}
By Fenchel's inequality, as well as the conditions,
\begin{gather}
\tilde\phi(1,\cdot)\geq -\tilde\psi, \quad \Dt\tilde\phi + H^*\left(\Dx\tilde\phi, \frac{1}{2}\Dxx\tilde\phi\right)\leq 0,\label{eqoptimumcondition2}
\end{gather}
each of the four terms in \eqref{eqoptimumcondition1} (in particular, terms inside the expectations and integrals) are non-negative. Hence they must all be bounded by $\epsilon$, which implies the required convergences \eqref{eqoptidual11}, as well as
\begin{gather}\label{eqoptimumcondition3}
0\leq\E^{\tilde \P} \bigg(\int_t^1 H(\tilde \alpha_s,\tilde \beta_s) + H^*\left(\Dx\phi^n_s, \frac{1}{2}\Dxx\phi^n_s\right)- \tilde\alpha_s\cdot \Dx\phi^n_s - \frac{1}{2}\tilde \beta_s :\Dxx\phi^n_s \,ds \bigg)\leq\epsilon. 
\end{gather}

If $H$ is strongly convex\footnote{Note that our definition of strongly convex does not require $H$ to be differentiable, since only subderivatives are used. Nevertheless, it implies that $H$ is strictly convex and thus $H^*$ is differentiable}, then $H^*$ is differentiable and we can define $(\alpha^n,\beta^n)$ such that
\[
(\alpha^n,\beta^n)=\nabla H^*\left(\Dx\phi^n, \frac{1}{2}\Dxx\phi^n\right), \quad \left(\Dx\phi^n, \frac{1}{2}\Dxx\phi^n\right)=\nabla H(\alpha^n,\beta^n).
\]
Hence, by the definition of convex conjugate and the strong convexity of $H$,
\begin{align}
&H(\tilde \alpha_s,\tilde \beta_s) + H^*\left(\Dx\phi^n_s, \frac{1}{2}\Dxx\phi^n_s\right) - \tilde\alpha_s\cdot \Dx\phi^n_s - \frac{1}{2}\tilde \beta_s :\Dxx\phi^n_s \nonumber\\
={}&H(\tilde \alpha_s,\tilde \beta_s) - H( \alpha_s^n, \beta_s^n) - (\tilde\alpha_s-\alpha_s^n)\cdot \Dx\phi^n_s - \frac{1}{2}(\tilde \beta_s-\beta_s^n) :\Dxx\phi^n_s \nonumber\\
\geq{}&C\left(\norm{\tilde\alpha_s-\alpha_s^n}^2+\norm{\tilde\beta_s-\beta_s^n}^2\right). \label{eqoptimumcondition4}
\end{align}
Combining \eqref{eqoptimumcondition3} and \eqref{eqoptimumcondition4} implies that $(\alpha^n,\beta^n)\to (\tilde \alpha, \tilde \beta)$ in $d\tilde\P\times dt$, completing the proof.
\end{proof}

\begin{remark}
If the optimum of the dual problem is obtained by a pair $(\tilde\psi,\tilde\phi)$, then all convergence results in Proposition \ref{propduality3} can be replaced by equalities.
\end{remark}

\subsection{Path dependent PDE}\label{secppde}

For this subsection as well as the next, we impose additional assumptions on $\psi$ and $H$, which is required for the well-posedness of the PPDE.
\begin{assumption}\label{asscostfuncionbb}
(i) For each $(\alpha,\beta)$, $H(\cdot,\cdot,\alpha,\beta)$ is $\FF$-progressively measurable.\\
(ii) The effective domain of $H$ is given by
$\operatorname{dom}H=\Lambda\times U$, where $U\subseteq \R^d\times \S^d_{++}$ is compact. In other words, for all $(t,\omega)\in\Lambda$, $H(t,\omega,\alpha,\beta)<+\infty$ if and only if $(\alpha,\beta)\in U$.\\
(iii) Furthermore, $H$ is bounded and uniformly continuous within its effective domain.\\
(iv) The function $\psi$ is bounded and uniformly continuous and has a common modulus of continuity with $H$.
\end{assumption}

\begin{remark}
Assumptions \ref{asscostfunction01} and \ref{asscostfuncionbb} have the following implications on $H^*(t,\omega,p,q)$:\\
(a) For fixed $(p,q)$, $H^*(\cdot,p,q)$ is $\FF$-progressively measurable and $|H^*(\cdot,p,q)|$ is bounded;\\
(b) $H^*$ is uniformly elliptic, i.e., there exists a constant $c>0$ such that $H^*(\cdot,q_1)-H^*(\cdot,q_2)\geq c\tr(q_1-q_2)$ for any $q_1\geq q_2$;\\
(c) $H^*$ is uniformly Lipschitz continuous in $(p,q)$;\\
(d) $H^*$ is uniformly continuous in $(t,\omega)$.

Condition (a) is satisfied since $H$ is progressively measurable and $H^*$ can be written as the supremum of a countable family of progressively measurable functions. Also note that $H^*(\cdot,0,0)$ is bounded. For (b), $H^*$ is indeed uniformly elliptic since $H(\cdot,\beta)$ is finite only if the eigenvalues of $\beta$ are uniformly bounded below by a positive constant. For (c), since the effective domain of $H(t,\omega,\cdot)$ is a bounded set $U$, $\nabla H^*(t,\omega,\cdot)$ must also be bounded. Finally, (d) holds because $H(\cdot,\alpha,\beta)$ is uniformly continuous on $\Lambda$.
\end{remark}

In order to obtain the uniqueness of viscosity solutions, we require an additional technical assumption  (see \cite{ekren2016viscosity2}, Assumption 3.5).
For all $\epsilon>0$, $n\geq 0$ and $0\leq T< T'\leq 1$, denote 
\[
\Pi^\epsilon_n(T,T'):=\{\pi_n=(t_i,x_i)_{1\leq i\leq n}: T<t_1<\cdots<t_n<T', |x_i|\leq \epsilon, 1\leq i\leq n \}.
\]
For all $\pi_n\in \Pi^\epsilon_n(T,T')$, let $\omega^{\pi_n}\in \Omega^T \cap \Omega_{T'}$ denote the linear interpolation of 
\[
(0,0),\ (T,0),\  (t_i,\sum_{j=1}^i x_j)_{1\leq i\leq n}\ \text{ and }\ (1,\sum_{j=n}^i x_j).
\]
\begin{assumption}\label{asscostfuncioncc}
There exists $0=T_0<\cdots<T_N=1$ such that the following holds for each $i=0,\ldots,N-1$. For any $\epsilon, n, p, q$, and $\omega\in\Omega_{T_i}+\Omega^{T_{i+1}}$, the functions $\pi_n \to \psi(\omega+\omega^{\pi_n})$ and  $\pi_n \to H^*(t, (\omega+\omega^{\pi_n})_{\cdot\wedge t},p,q)$ are uniformly continuous in $\Pi^\epsilon_n$.
\end{assumption}

Now we will present Proposition \ref{thmsummary01} (iii), which is restated here.
\begin{proposition}\label{thmppdesoln01}
Suppose that Assumptions \ref{asscostfunction01} and \ref{asscostfuncionbb} are satisfied. Then $J_{\psi,H}$ has the representation
\begin{align}\label{eqppde101}
J_{\psi,H}(t,\rho_t)=\sup_{\P\in\Pset^0_t(\rho_t)} -\E^\P \psi-\EP \int_t^1 H(\alpha^\P_s,\beta^\P_s) \,ds=\int_{\Omega_t} \phi_{\psi,H}(t,\cdot)\, d\rho_t.
\end{align}
where $\phi_{\psi,H}$ is a viscosity solution of the following path dependent PDE:
\begin{gather}\label{eqppde10}
\phi(1,\cdot)= -\psi \quad \text{and}  \quad \Dt\phi + H^*\left(\Dx\phi, \frac{1}{2}\Dxx\phi\right)= 0.
\end{gather}
Moreover, if Assumption \ref{asscostfuncioncc} is also satisfied, then $\phi_{\psi,H}$ is the unique viscosity solution of the PPDE \eqref{eqppde10}.
\end{proposition}
\begin{proof}
Following Proposition \ref{corduality2}, it suffices to show that the solution the control problem
\[
J_{\psi,H}(t,\delta(x_{\cdot\wedge t}))=\sup_{\P\in\Pset^0_t(\delta(x_{\cdot\wedge t}))} -\E^\P \psi-\EP \int_t^1 H(\alpha^\P_s,\beta^\P_s) \,ds
\]
is a viscosity solution of the PPDE \eqref{eqppde10}. This can be proven by following the same argument of \cite{ekren2016viscosity1}, Proposition 4.7, which is an adaptation of the standard argument for deriving HJB equations in the PPDE settings.
%
If Assumption \ref{asscostfuncioncc} also holds, then we have the full comparison principle for the PPDE \eqref{eqppde10} and $\phi_{\psi,H}$ must be the unique viscosity solution (see \cite{ekren2016viscosity2}, Theorem 4.1).
\end{proof}

\begin{remark}
As mentioned in \cite{ekren2016viscosity2} Section 8.3, Assumption \ref{asscostfuncioncc} is a technical assumption only used for the construction of a family of viscosity solutions to an approximating family of path-frozen PDEs. It is likely that in control problems such as the one in Proposition \ref{thmppdesoln01}, one can directly construct the required family of viscosity solutions via a sequence of approximating control problems and their HJB equations, without the need for Assumption \ref{asscostfuncioncc}. 

Another approach would be to utilise the comparison principle result of \cite{ren2017comparison}, which requires a different metric on $\Lambda$ in the definition of uniform continuity. It also does not require Assumption \ref{asscostfuncioncc}. We will defer these potential approaches to future research, and refer the interested readers to  \cite{ekren2016viscosity1, ekren2016viscosity2, ren2017comparison} for detailed discussions of viscosity solutions to fully nonlinear PPDEs and various versions of regularity assumptions.
%
\end{remark}

\subsection{Dimension reduction}\label{seclocalisation}
One particular interesting case is when $\rho_t$ is a Dirac measure, i.e., $\rho_t=\delta_{t,x}=\delta(X_{\cdot\wedge t}=x_{\cdot\wedge t})$. Then
\begin{gather}\label{eqjj11}
J_{\psi,H}(t,\delta_{t,x})=\sup_{\P\in\Pset^0_t(\delta_{t,x})}-\E^\P \psi-\EP \int_t^1 H(\alpha^\P_s,\beta^\P_s) \,ds
=\inf_{\phi\in C^{1,2}_t(\Lambda)}  \phi(t,x_{\cdot\wedge t}),\\
\text{subject to}\quad 
\phi(1,\cdot)\geq -\psi \quad \text{and}  \quad \Dt\phi + H^*\left(\Dx\phi, \frac{1}{2}\Dxx\phi\right)\leq 0.\label{eqjj12}
\end{gather}
For $x\in\Omega$, let us use the shorthand $J(t,x):=J_{\psi,H}(t,\delta_{t,x})$.

In general, $J(t,\cdot)$ only depends on the path up to time $t$ and can be shown to be $\Filt_t$-measurable. But in all practical examples, both the payoff function $\psi$ and the cost function $H$ only depend on certain features of the path rather than the entire path, and can often be parametrised by a finite number of state variables. Intuitively, the solution $J(t,\cdot)$ should be Markovian with respect to in terms of those state variables. This is useful in practice since it allows us to identify the state variables driving the relevant features of the path, and reduces an infinite dimensional PPDE to a finite dimensional PDE.

This idea of dimension reduction is well-known in numerical methods for BSDEs (see e.g., \cite{gobet2005regression,zhang2004numerical}) and path-dependent stochastic control problems (see e.g., \cite{tan2014discrete}). The usual approach is to identify an ``updating function'' \cite{brunick2013mimicking} or equivalent which updates the state variables according the evolution of the paths. In this paper, we will formalise this notion in a slightly different approach. Here, we introduce a so-called ``semifiltration'', which is an appropriate class of $\sigma$-algebra smaller than $\FF$ that captures the information required for the solution. Typically, this can simply be the $\sigma$-algebras generated by the state variables, but our construction also allows for more general cases. Moreover, as shown in Remark \ref{remminimalsemifiltration}, our approach enables the explicit construction of the ``minimal'' semifiltration for an arbitrary problem.

\begin{definition}[Semifiltration]\label{deflocfilt}
(i) A collection of $\sigma$-algebras $\GG=\{\Gfilt_t: t\in[0,1]\}$ is called a \emph{semifiltration} if the following properties hold:
\begin{itemize}
\item for every $t$, $\Gfilt_t \subseteq \Filt_t$;
\item for all $t<u$, $\Gfilt_u \subseteq \Gfilt_t \vee \Filt_{[t,u]} $, where $\Filt_{[t,u]}$ is a $\sigma$-algebra defined by
\[
\Filt_{[t,u]}= \sigma(\{X_s-X_t: s\in[t,u]\}).
\]
\end{itemize}
(ii) A pair of functions $(\psi, H) \in C_b(\Omega)\times C_b(\Lambda\times U)$ is said to be \emph{adapted} to a semifiltration $\GG$ if $\psi$ is $\Gfilt_1$-measurable and for every $t\in[0,1]$, $H(t,\cdot,\cdot)$ is $\Gfilt_t \otimes \Borel(U)$-measurable.

\end{definition}

In general, a semifiltration $\GG$ is not a filtration. 
It has the following interpretation. As we move forward in time, $\GG$ collects more information about the canonical process $X$, in the same way that $\FF$ does. At the same time, $\GG$ is also allowed to ``forget'' information, which occurs whenever the inclusion $\Gfilt_u \subseteq \Gfilt_t \vee \Filt_{[t,u]} $ is strict. Moreover, $\GG$ is not allowed to ``recall'' information once it is ``forgotten''.

For the numerical implementation of many practical problems, $\GG$ has the advantage of being much smaller than $\FF$. For example, if the problem is known to be Markovian, then it suffices to keep track of the current value of the state variable $\Gfilt_t=\sigma(X_t)$, as opposed to the history of the entire path $\Filt_t$. Ideally, we would choose $\GG$ to be as small as possible, while still retaining all dependent variables required for the solution. 

The condition $\Gfilt_u \subseteq \Gfilt_t \vee \Filt_{[t,u]} $ can be interpreted as a time consistency condition on semifiltrations. Given $\Gfilt_t$ and all possible information on the time period $[t,u]$, we have enough to construct $\Gfilt_u$. This notion is formalised in the following crucial lemma.
\begin{lemma}\label{lemlocalfilt}
Let $\GG$ be a semifiltration and $\XX, \YY$ be Polish spaces. Fix $t\in[0,1]$ and consider a function $f\in C_b(\Omega\times \XX, \YY)$ that is $((\Gfilt_t\vee\Filt_{[t,1]})\otimes \Borel(\XX), \Borel(\YY))$-measurable. Define the map $\Gamma:\Omega\to C_b(\Omega^t\times \XX; \YY)$ via
$$\omega \stackrel{\Gamma}{\rightarrow} f(\omega_{\cdot\wedge t}+ \cdot, \cdot ).$$
Then $\Gamma$ is $\Gfilt_t$-measurable.
\end{lemma}

\begin{proof}
Note that the space $C_b(\Omega^t\times \XX; \YY)$ is endowed with the supremum norm $\norm{\cdot}_\infty$ and the associated Borel $\sigma$-algebra.

For any open ball $D\subset C_b(\Omega^t\times \XX; \YY)$ with centre $g$ and radius $r$, we have to check that $\Gamma^{-1}(D)\in \Gfilt_t$. Due to the separability of $\Omega^t\times \XX$, we can write
\begin{gather}\label{eqlocalfiltlem1}
\Gamma^{-1}(D) = \bigcap_{(\omega',x)\in \mathcal S} \Gamma_{(\omega',x)}^{-1}(D_{(\omega',x)}),
\end{gather}
where $\mathcal S$ is a countable dense subset of $\Omega^t\times \XX$, the map $\Gamma_{(\omega',x)}:\Omega \to \YY$ is defined by
\[
\omega \to f(\omega_{\cdot\wedge t}+\omega',x)
\]
and $D_{(\omega',x)}\subset\YY$ is an open ball with centre $g(\omega',x)$ and radius $r$. By \eqref{eqlocalfiltlem1}, it suffices to check that $\Gamma_{(\omega',x)}^{-1}(D_{(\omega',x)})\in \Gfilt_t$ for all $(\omega',x)\in\Omega^t\times \XX$.

Now since $f^{-1}(D_{(\omega',x)})\in (\Gfilt_t\vee\Filt_{[t,1]})\otimes \Borel(\XX)$, by taking the $x$-section of this set, we have $f^{-1}(D_{(\omega',x)},x)\in \Gfilt_t\vee\Filt_{[t,1]}$. Then we have the following implications
\begin{align}
&\{\omega:f(\omega,x)\in D_{(\omega',x)}\}\in \Gfilt_t\vee\Filt_{[t,1]} \nonumber\\
\implies\quad &\{\omega:f(\omega,x)\in D_{(\omega',x)}, \omega-\omega_{\cdot\wedge t}=\omega' \}\in \Gfilt_t\vee\Filt_{[t,1]}\nonumber\\
\iff\quad &\{\omega:f(\omega_{\cdot\wedge t}+\omega',x)\in D_{(\omega',x)}, \omega-\omega_{\cdot\wedge t}=\omega' \}\in \Gfilt_t\vee\Filt_{[t,1]} \label{eqlocalfiltlem2}\\ 
\implies\quad &\{\omega:f(\omega_{\cdot\wedge t}+\omega',x)\in D_{(\omega',x)} \}\in \Gfilt_t \label{eqlocalfiltlem3}.
\end{align}
The last implication relies on the fact that there is a natural isomorphism
$$\Gfilt_t \vee \Filt_{[t,1]}\cong\bar\Gfilt_t \otimes \Borel(\Omega^t),$$
where $\bar\Gfilt_t$ is a $\sigma$-algebra on $\Omega_t$. Since $f(\omega_{\cdot\wedge t}+\omega',x)\in D_{(\omega',x)}$ does not depend on the path after time $t$, \eqref{eqlocalfiltlem3} is obtained by sectioning \eqref{eqlocalfiltlem2} by $\omega-\omega_{\cdot\wedge t}=\omega'\in\Omega^t$.
Since \eqref{eqlocalfiltlem3} is equivalent to $\Gamma_{(\omega',x)}^{-1}(D_{(\omega',x)})\in \Gfilt_t$, this completes the proof.
\end{proof}

Now we are in a position to prove Theorem \ref{thmsummary01} (iv), which is restated below as Proposition \ref{thmlocalise}.
\begin{proposition}\label{thmlocalise}
Suppose that Assumptions \ref{asscostfunction01} and \ref{asscostfuncionbb} are satisfied. Let $(\psi,H)$ be adapted to a semifiltration $\GG=\{\Gfilt_t: t\in[0,1]\}$.
The we have the following results:
\\
(i) The map $J(t,\cdot)$ defined by \eqref{eqjj11} is $\Gfilt_t$-measurable;\\
(ii) If $\phi\in C^{1,2}_t(\Lambda)$ is a classical solution to the PPDE \eqref{eqppde10} on $[t,1]$, and $H$ is strictly convex in $(\alpha,\beta)$, then the pair $(\alpha,\beta)= \nabla H^*\left(\Dx\phi(t,\cdot), \frac{1}{2}\Dxx\phi(t,\cdot)\right)$ is $\Gfilt_t$-measurable.
\end{proposition}
\begin{proof}
(i) First note that, since $(\psi,H)$ is bounded, continuous and adapted to $\GG$, the map $(\omega,u,v)\to (\psi(\omega),H(u,\omega_{\cdot\wedge s},v))$ where $u\geq t$ is bounded, continuous and $(\Gfilt_t \vee \Filt_{[t,1]})\otimes \Borel(\R\times U)$ measurable. Hence by Lemma \ref{lemlocalfilt},
the map $\Gamma:\Omega\to C_b(\Omega^t)\times C_b(\Lambda^t\times U)$ defined by
$$\omega\stackrel{\Gamma}{\to} (\psi(\omega_{\cdot\wedge t}+ \cdot ),H(\cdot, \omega_{\cdot\wedge t}+ \cdot, \cdot)).$$
is $\Gfilt_t$-measurable.

After a suitable translation $J(t,x)$ can be written as
\begin{align*}
J(t,x)&=\sup_{\Pset^0_t(\delta_{t,x})}-\E^\P \psi-\EP \int_t^1 H(s,\cdot,v^\P_s) \,ds\\
&=\sup_{\Pset^0_t(\delta_{t,0})}-\E^\P \psi(x+\cdot)-\EP \int_t^1 H(s,x+\cdot,v^\P_s) \,ds.
\end{align*}
Hence $J(t,x)$ can be written as the composition
\[
x \stackrel{\Gamma}{\to} (\psi, H) \to \sup_{\Pset^0_t(\delta_{t,0})}-\E^\P \psi-\EP \int_t^1 H(\cdot,v^\P_s) \,ds.
\]
We have already established that $\Gamma$ is $\Gfilt$-measurable. The second map is continuous, since both $\E^\P$ and $\E^\P\int_t^1 \cdot \, ds$ are Lipschitz continuous functionals. Therefore $J(t,\cdot)$ is $\Gfilt_t$-measurable.

(ii) First, the $\Gfilt_t$-measurability of $\phi$ follows immediately from Proposition \ref{thmppdesoln01}. 
By the functional It\^o formula, for $u> t$
\begin{gather}\label{eqderivmeasito}
\phi(u,X)-\phi(t,x)=\int_t^u \Dt \phi\, ds+ \Dx \phi \cdot dX_s + \frac{1}{2} \Dxx \phi : d\langle X\rangle_s, \quad \P\text{-\as}
\end{gather}
holds for all $\P\in\Pset^1_t$. Consider the probability $\P\in\Pset^1_t(\delta_{t,x})$ whose characteristics $(g,h)$ is constant on $[t,1]$. Then the infinitesimal generator of $\phi(t,X)$ is given by
\begin{gather}\label{eqderivmeas}
\lim_{u \searrow t} \frac{\E^\P(\phi(u,X))-\phi(t,x)}{u-t}=\Dt \phi(t,x) + \Dx \phi(t,x) \cdot g + \frac{1}{2} \Dxx \phi(t,x) : h.
\end{gather}
By Lemma \ref{lemlocalfilt} and Proposition \ref{thmlocalise} (i), the left hand side of \eqref{eqderivmeas}, as a function of $x$, is $\Gfilt_t$-measurable. Since \eqref{eqderivmeas} holds for all $(g,h)\in \R^d\times \S^d_+$, $(\Dt\phi(t,x), \Dx \phi(t,x),  \Dxx \phi(t,x))$ must also be $\Gfilt_t$-measurable. Finally, $H$ being strictly convex on $U$ implies that $\nabla H^*$ is continuous. Thus $(\alpha_t,\beta_t)= \nabla H^*\left(\Dx\phi_t(t,\cdot), \frac{1}{2}\Dxx\phi_t(t,\cdot)\right)$ is also $\Gfilt_t$-measurable.
\end{proof}

\begin{remark}\label{remminimalsemifiltration}
In fact, there exists a ``minimal'' semifiltration $\tilde\GG$ to which $(\psi,H)$ is adapted. For each $t$, we can define a $\sigma$-algebra $\tilde \Gfilt_t=\sigma(\Gamma)$, where $\Gamma_t$ is the map defined by,
$$\omega\stackrel{\Gamma_t}{\to} (\psi(\omega_{\cdot\wedge t}+ \cdot ),H(\cdot, \omega_{\cdot\wedge t}+ \cdot, \cdot)).$$
From the proof of Proposition \ref{thmlocalise}, we have seen that $\tilde\Gfilt_t\subseteq\Gfilt_t$ for every semifiltration $\GG=\{\Gfilt_t,t\in[0,1]\}$ to which $(\psi,H)$ is adapted. So it suffices to show that $\tilde\GG=\{\tilde\Gfilt_t,t\in[0,1]\}$ is indeed a semifiltration. This reduces to checking that $\Gamma_u$ is $\tilde\Gfilt_t\vee\Filt_{[t,u]}$-measurable for $t<u$, which follows from the continuity of the map $(\Gamma_t, \omega_{[t,u]})\to \Gamma_u$.
\end{remark}

\section{Volatility calibration}

The  results of this paper can be applied to the problem of calibrating a volatility model to complex path dependent derivatives. Without loss of generality, let us assume that the interest rate is 0.  Suppose that some derivative prices are known, the goal is to find a martingale diffusion for the underlying asset which attains those prices. As far as the authors are aware of, exact calibration techniques are only available on vanilla products (European options). The path dependent nature of the results presented allows us to include path dependent products such as Asian options, barrier options and lookback options. 

Let the canonical process $X$ be the logarithm of the underlying stock price. Note that the logarithm transform is purely chosen for notational and numerical convenience, and is not at all necessary. We are interested in finding a probability measure $\P\in\Pset^0$ with characteristics $(-\frac{1}{2}\sigma^2, \sigma^2)$ where $\sigma$ is some $\FF$-adapted process. In other words, we want $X$ to be an $(\FF,\P)$-semimartingale in the form of
\[
dX_t=-\frac{1}{2}\sigma^2 dt + \sigma dW^\P_t.
\]
Next let $G : \Omega \to R^m$ denote a vector of $m$ (path dependent) discounted option payoff functions. We further restrict $\P$ so that the options have known prices $\E^\P G=c$ for some $c\in\R^m$. It is immediate that this problem is a special case of the general problem we introduced in Section \ref{sec2}, specifically in Example \ref{examplea03}. In particular, we want to solve
\begin{gather*}
\inf_{\P\in\Pset^0_t(\delta_{X_0})} \EP \int_0^1 H(\alpha^\P_s,\beta^\P_s) \,ds,\quad
\text{subject to}\quad \E^\P G=c,
\end{gather*}
where $H$ is any suitable convex cost function whose effective domain is within the set $\{-2\alpha=\beta\}$. In the examples of this Section, we consider a cost function of the form
\[
H(\alpha,\beta)=
\begin{cases}
a(\beta/ \bar\sigma^2)^p+b(\beta/ \bar\sigma^2)^{-q}+c, & -2\alpha= \beta,\\
\infty,& -2\alpha\neq \beta,
\end{cases}
\]
where $\bar\sigma$ is some reference volatility level, $p, q$ are constants greater than 1, and $a, b, c$ are constants chosen so that the function reaches its minimum at $\beta=\bar\sigma^2$ with $\min H=0$. The basic idea is to keep $\beta$ positive and penalise any large deviations from $\bar\sigma^2$.\footnote{The results from Sections \ref{seclocalisation} and \ref{secppde} required the cost function $H$ to have a compact effective domain, this can be achieved by truncating $H$ at extreme values of $\beta$ and setting $H$ to infinity outside of these values.}

As mentioned in Example \ref{examplea03}, the corresponding saddle point problem is 
\begin{gather*}
V=\inf_{\P\in\Pset^0_t(\delta_{X_0})}\sup_{\lambda \in\R^m} \lambda\cdot( \EP G-c)+ \EP\int_0^1 H(\alpha^\P_s,\beta^\P_s) \,ds.
\end{gather*}
Applying Theorem \ref{thmsummary01} and assuming sufficient regularity on the payoff functions, the dual formulation of the problem is
\begin{align}\label{eqnumerdual}
V=\VV=\sup_{\lambda \in\R^m} -\lambda\cdot c - \phi(0,X_0),
\end{align}
where $\phi$ is a solution to the PPDE
\begin{align}\label{eqnumerppde}
\phi(1,\cdot)= -\lambda\cdot G \quad \text{and}  \quad \Dt\phi + H^*\left(\Dx\phi, \frac{1}{2}\Dxx\phi\right)= 0.
\end{align}

Numerically, the difficult part is to solve the PPDE \eqref{eqnumerppde} to find $\phi(0,X_0)$. The key idea is to use Theorem \ref{thmsummary01} (iv) to effectively reduce the dimensionality of $\phi$ to a manageable size. For many examples, instead of being dependent on the whole path of $X$, the solution $\phi$ is in fact Markovian with respect to a few state variables. 
The general abstract result is described in Definition \ref{thmlocalise}. But for most practical cases, it is straightforward to identify the relevant state variables by inspection and they are the familiar variable from standard option pricing techniques. Here are some examples of relevant state variables:
\begin{itemize}
\item European options: the spot price $X_t$;
\item Asian options: the spot price $X_t$ and the running average $\frac{1}{t}\int_0^t X_t\, dt$;
\item Continuous barrier options: the spot price $X_t$ and the indicator variable $\I(X_s>B, s\in[0,t])$ for lower barriers or $\I(X_s<B, s\in[0,t])$ for upper barriers;
\item Lookback options: the spot price $X_t$ and either the running minimum $\min_{s\in[0,t]} X_s$ or running maximum $\max_{s\in[0,t]} X_s$.
\end{itemize}
In each case, the relevant semifiltration at time $t$ is the $\sigma$-algebra generated by each set of state variables at time $t$.
Then the PPDE reduces to a PDE which depends on the spot price as well as an additional path dependent variable and can be solved via conventional methods.

The supremum in \eqref{eqnumerdual} over $\lambda\in\R^m$ is computed by a standard optimisation routine. This process can be further aided by numerically computing the gradient of the objective with respect to $\lambda$ in the following way. From \eqref{eqnumerppde}, $\phi$ is also a function a $\lambda$ via the terminal condition. Writing $\phi'=\partial_\lambda \phi$ and differentiating the PPDE with respect to $\lambda$, we obtain
\begin{align}\label{eqnumerppde20}
\Dt\phi' + \alpha\Dx\phi' + \frac{1}{2}\beta\Dxx\phi' = 0.
\end{align}
where $(\alpha,\beta)=\nabla H^*\left(\Dx\phi, \frac{1}{2}\Dxx\phi\right)$.
Hence the gradient of the objective function is given by $\nabla_\lambda (-\lambda\cdot c - \phi(0,X_0)) = - c-\phi'(0,X_0)$ where $\phi'$ satisfies $\phi'(1,\cdot)=-G$ and \eqref{eqnumerppde20}. 
Once the optimal $\lambda$ and $\phi$ have been found, the optimal volatility is given by $\sigma^2=\beta=-2\alpha$ where $(\alpha,\beta)=\nabla H^*\left(\Dx\phi, \frac{1}{2}\Dxx\phi\right)$.

\begin{remark}
(i) The gradient of the objective function with respect to $\lambda$ has a natural financial interpretation. Since $\phi'(0,X_0) = \EP(-G)$ where $\P$ has characteristics $(\alpha,\beta)$, the gradient is in fact $\EP G - c$, or the difference between the option prices given by the current optimisation iteration and the target option prices. The optimum is reached when that difference is zero, in other words, the target option prices are attained exactly. 

(ii)
Recall that in our original formulation, $G$ is required to be a bounded and uniformly continuous function. In practice, many options do not have bounded payoffs (e.g., call options). This can be fixed by either converting them into options with bounded payoffs via arbitrage arguments (e.g., put options via put-call parity), or by truncating the domain at some extremely large value. Option that do not have continuous payoffs (e.g., digital options, barrier options) can be approximated by uniformly continuous functions.

(iii)
By increasing the dimension of the canonical process to include other features such as the variance process, the same technique can be used to calibrate local stochastic volatility (LSV) models. See \cite{guo2019calibration} for more details.

(iv) Based on the result of this paper, similar methods have been developed for the calibration of LSV models \cite{guo2019calibration} and the joint calibration of stock and VIX options \cite{guo2020joint}.

(v) Our results here are limited to non-callable products, hence excluding the calibration of Bermudan and American options. Callable products involve incorporating stopping times into the duality spaces, which requires additional techniques beyond the scope of the current work. This problem is under current research and will be addressed in a forthcoming paper.
\end{remark}

In the following subsections, we will demonstrate a few numerical examples that involve calibrating volatility functions to European, Barrier and Lookback options. Since each of the payoff functions considered satisfies Assumption \ref{asscostfuncioncc} (see also, \cite{ekren2016viscosity2} Lemma 3.6), the corresponding PPDE \eqref{eqnumerppde} has a unique viscosity solution. For the the convergence of the numerical schemes for viscosity solutions to PPDEs, we refer to \cite{ren2017convergence, zhang2014monotone}.

\subsection{European options}
European options have payoff functions of the form $G(X_T)$ where each option depends on the value of the underlying at a fixed maturity $T$.
In this case, the function $\phi$ and the optimal volatility $\sigma$ only depend on the state variable $t$ and $X_t$. In other words, we recover a local volatility model. This is consistent with classical approaches such as Dupire's formula \cite{dupire1994pricing}. In some sense, local volatility models are the ``simplest'' models that can calibrate to all European products. Unlike Dupire's formula, our approach does not require the interpolation of option prices between discrete strikes and maturities.
The functional derivatives in the PPDE simply reduces to the usual partial derivatives,
\begin{align}\label{eqeurpde}
\partial_t\phi + H^*\left(\partial_x\phi, \frac{1}{2}\partial^2_x\phi\right)= 0.
\end{align}
Figure \ref{figeur} shows an example of a volatility calibrated to European options at all strikes and four different maturities.
\begin{figure}[!h]
\includegraphics[width=\textwidth,trim={3cm 2cm 3cm 2cm},clip]{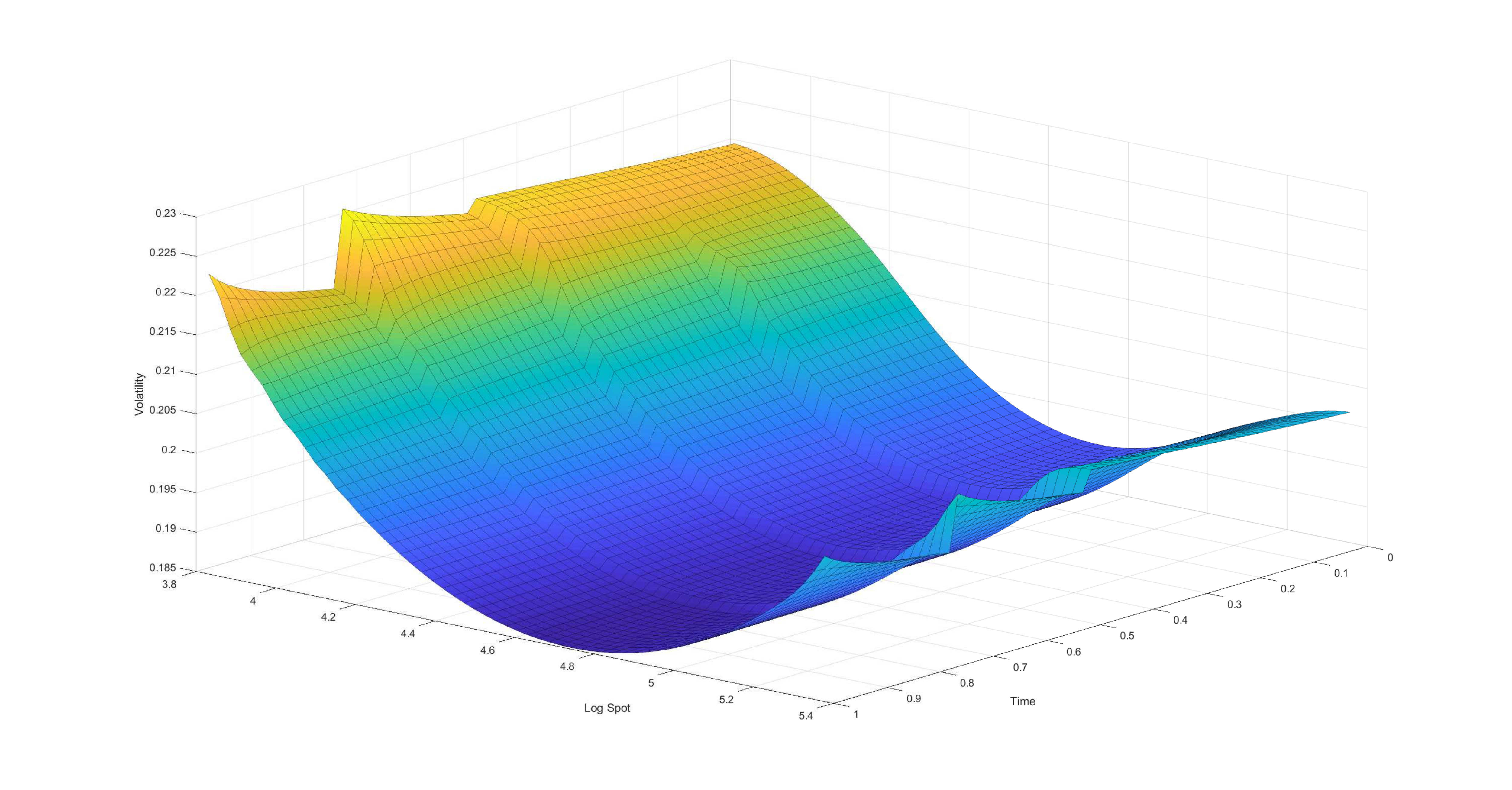}
\caption{Volatility surface calibrated to European put options at all strikes and four different maturities}
\label{figeur}
\end{figure}

\subsection{Barrier options}

\begin{figure}[!ht]
\includegraphics[width=\textwidth,trim={3cm 2cm 3cm 2cm},clip]{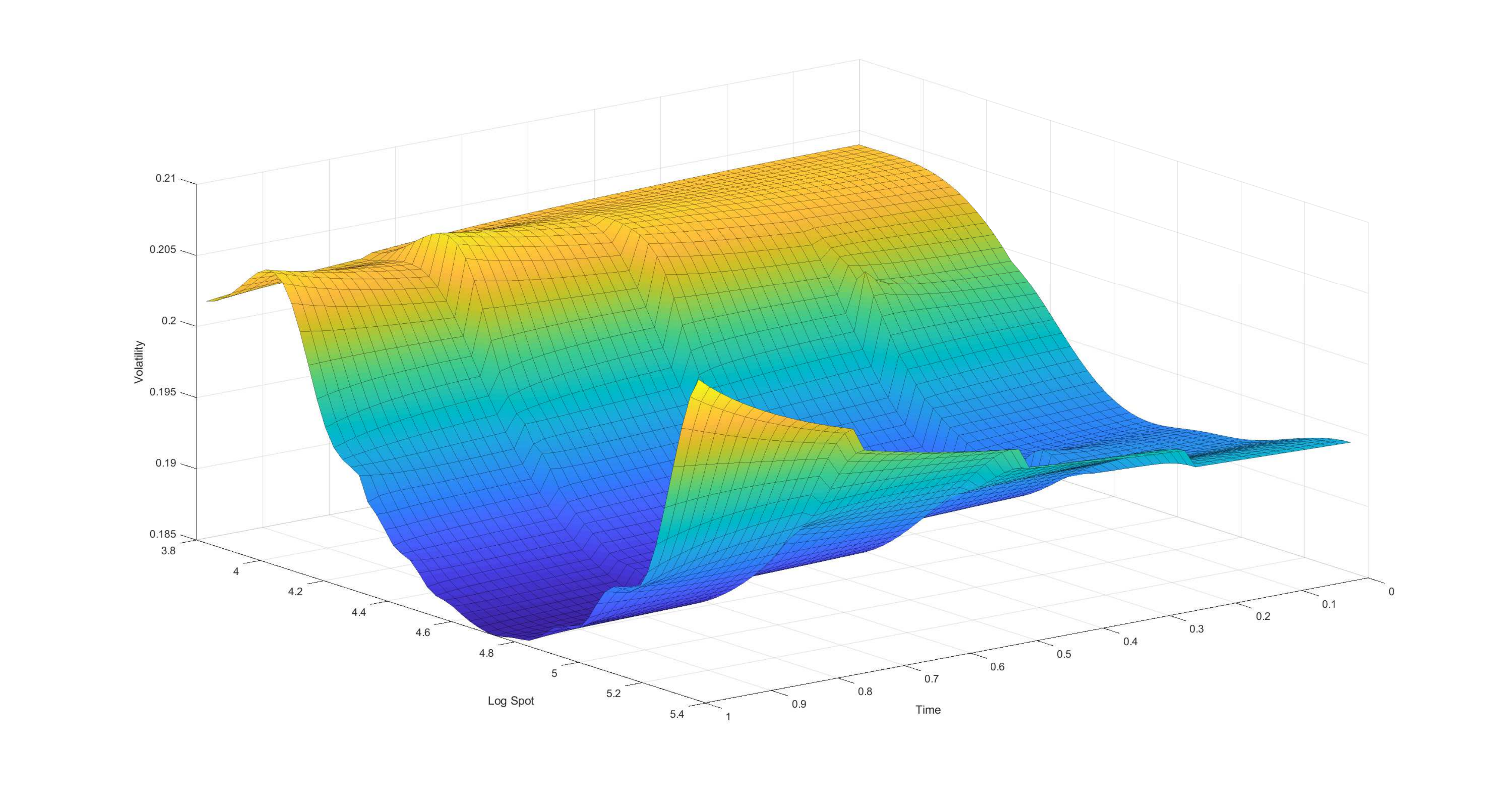}
\includegraphics[width=\textwidth,trim={3cm 2cm 3cm 2cm},clip]{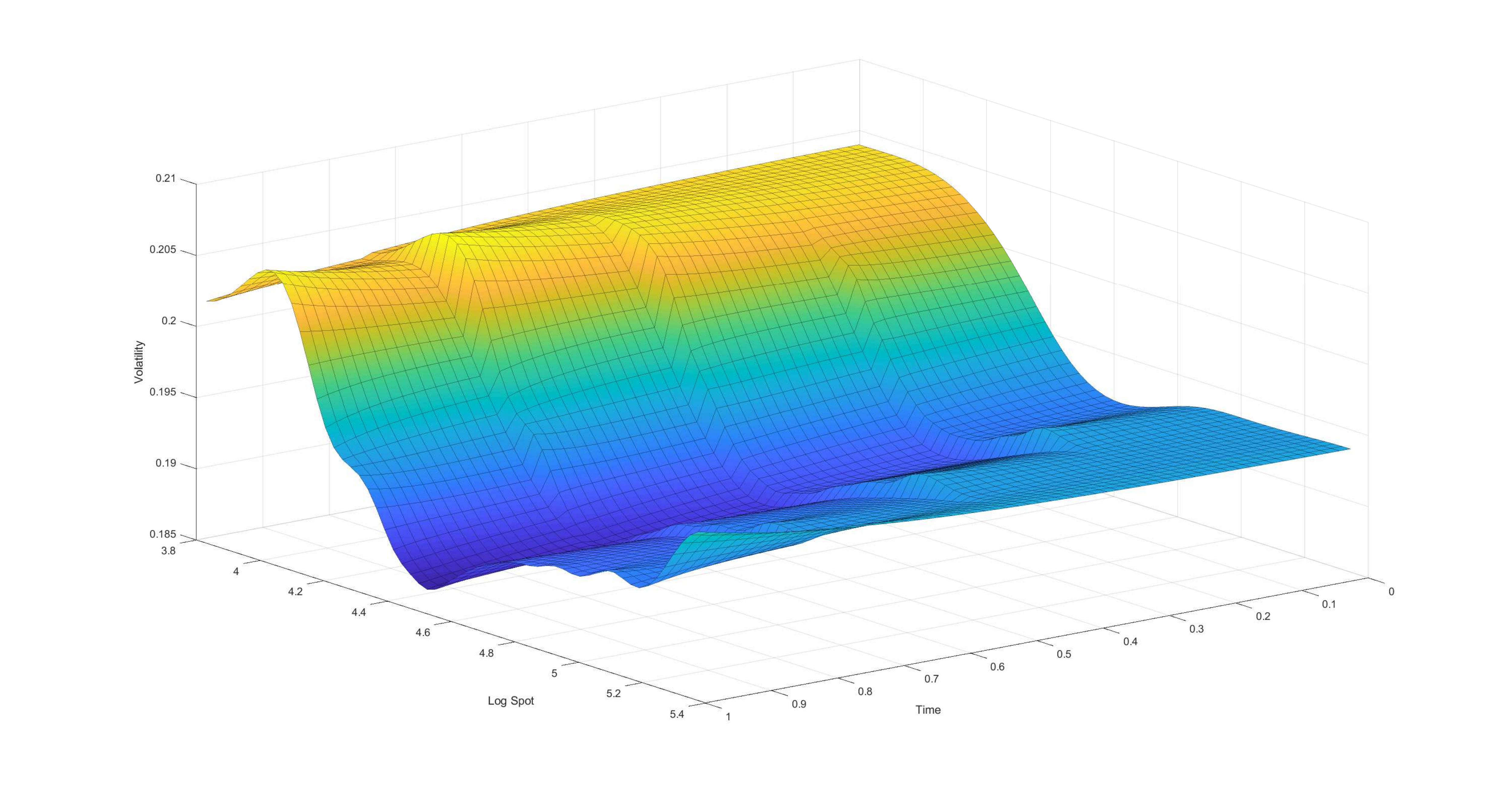}
\caption{Volatility calibrated to all down-and-in and down-and-out puts at all strikes and four different maturities. The top half is showing $\sigma_0$ (before hitting the barrier) and the bottom half is showing $\sigma_1$ (after hitting the barrier).}
\label{figbar}
\end{figure}

Formally speaking, a barrier is a closed subset $\mathbb B \subset [0,1]\times \R$ whose complement is a connected region containing $(0,X_0)$. 
The payoff of a barrier product expiring at time $T$ is a function of $X_T$ and the indicator variable $\I_T:=\I(X_s\in \mathbb B, \text{for some } s\in[0,t])$, checking whether the path of the underlying has hit the barrier. 
When calibrating to a collection of barrier products with a single fixed barrier, the required state variables are $t, X_t$ and $\I_t$. Then the function $\phi$ can be effectively split into two functions, $\phi^0(t,x)$ and $\phi^1(t,x)$, corresponding to the cases $\I_t=0$ and $\I_t=1$, respectively. The PDE is then given by
\begin{align*}
\partial_t\phi_1 + H^*\left(\partial_x\phi_1, \frac{1}{2}\partial^2_x\phi_1\right)= 0,& \quad (t,x)\in [0,1]\times \R,\\
\partial_t\phi_0 + H^*\left(\partial_x\phi_0, \frac{1}{2}\partial^2_x\phi_0\right)= 0,& \quad (t,x)\notin \mathbb B,\\
\phi_0 = \phi_1,& \quad (t,x)\in \partial\mathbb B.
\end{align*}
Similarly, the optimal volatility will be switching between two local volatilities $\sigma_0(t,x)$ and $\sigma_1(t,x)$, conditional to whether the underlying has hit $\mathbb B$ or not. The PDE for $\phi_0$ will be used to compute the volatility function prior to the stock hitting the barrier, while the PDE for $\phi_1$ will be used to compute the volatility function after the barrier has been hit.

If we calibrate to options with $l$ distinct barriers, then a similar approach applies but with $l$ indicator variables. In this case $\phi$ and $\sigma$ will be split into $2^l$ functions, conditioning on the subset of the barriers that has been reached. The number $2^l$ can be reduced in many cases by eliminating combinations of barrier events are not reachable. For example, if the barriers are nested, then only $l+1$ functions are needed.

As an example, let us consider barrier products with respect to a continuous lower barrier $\{x\leq b\}$ where $b<X_0$ is a constant. In particular, we will be calibrating to all down-and-in and down-and-out puts at all strikes and four different maturities. The top half of Figure \ref{figbar} shows the calibrated volatility function $\sigma_0$ (before hitting the barrier) and the bottom half shows $\sigma_1$ (after hitting the barrier). Even though $\sigma_0$ is only defined for $x\geq b$, for the purpose of visualisation, we set $\sigma_0 = \sigma_1$ for $x<b$.

\subsection{Lookback options}
\begin{figure}[!htp]
\includegraphics[width=\textwidth]{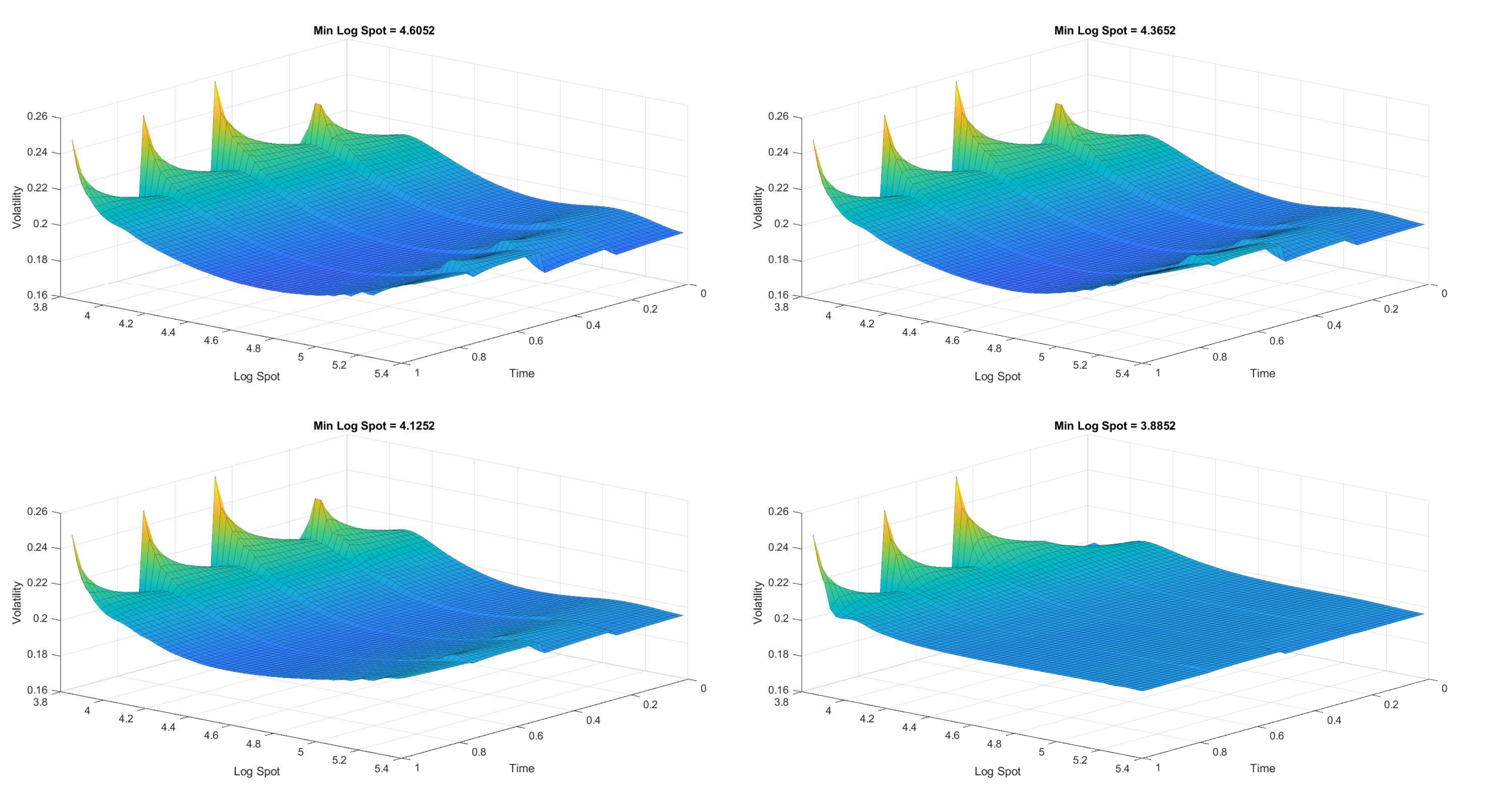}
\includegraphics[width=\textwidth]{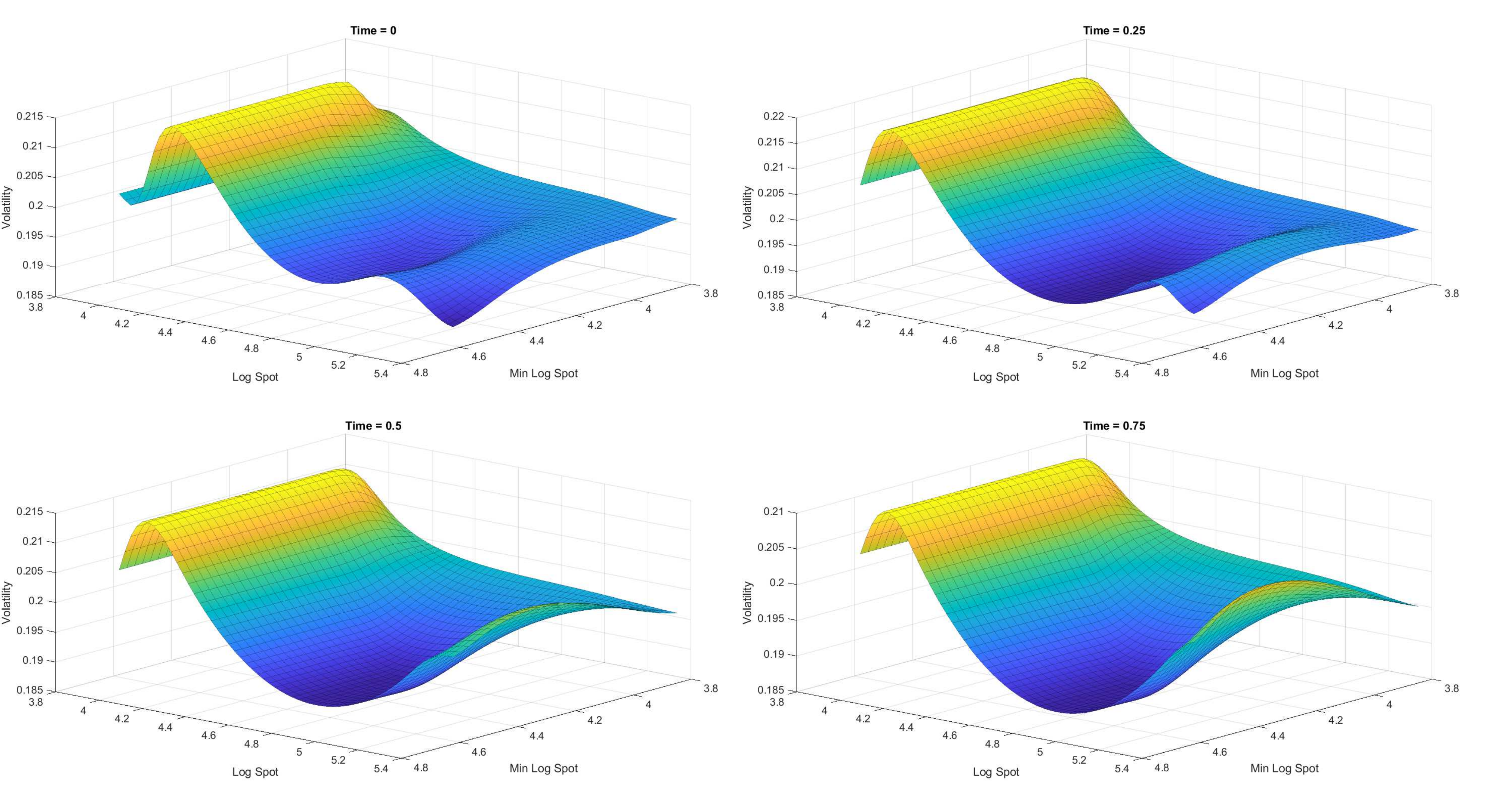}
\caption{Volatility function $\sigma(t,x,y)$ calibrated to European puts, all lower barrier down-and-out puts and fixed strike lookback puts at all strikes and four different maturities. The top half of the figure shows cross sections at different values of $t$ while the bottom half shows cross sections at different values of $y$.}
\label{figlookback}
\end{figure}
The payoff of lookback products expiring at time $T$ depends on $X_T$ as well as either $\max_{s\in[0,T]} X_s$ or $\min_{s\in[0,T]} X_s$. Here we will focus on cases involving the minimum $Y_T:=\min_{s\in[0,T]} X_s$.   For example, the payoff of a fixed strike lookback put with strike $K$ is given by $(K-Y_T)^+$.
In this case, the state variables for $\phi$ and $\sigma$ will be $t, X_t$ and $Y_t$. Note that European options and barrier options with lower barriers are also special cases of lookback products.
The PDE is then given by
\begin{align}\label{eqlookbackpde}
\partial_t\phi + H^*\left(\partial_x\phi, \frac{1}{2}\partial^2_x\phi\right)= 0,& \quad x\geq y,\\
\partial_y\phi=0,& \quad x=y.\label{eqlookbackpde2}
\end{align}
The boundary condition \eqref{eqlookbackpde2} can be obtained in the following way. First, via standard arguments using the dynamic programming principle, and the fact that $Y$ has finite variation, we derive the following equation
\[
\partial_t\phi\, dt + H^*\left(\partial_x\phi, \frac{1}{2}\partial^2_x\phi\right)dt+\partial_y\phi \, dY_t=0.
\]
Then, by using the argument from \cite{privault2013stochastic} Proposition 8.5., the required boundary condition \eqref{eqlookbackpde2} follows from the fact that $dt$ and $dY_t$ are mutually singular.

In Figure \ref{figlookback}, we show the resulting volatility function $\sigma(t,x,y)$ calibrated to European puts, all lower barrier down-and-out puts and fixed strike lookback puts at all strikes and four different maturities. The top half of the figure shows cross sections at different values of $t$ while the bottom half shows cross sections at different values of $y$. Even though $\sigma$ is only defined for $x\geq y$, for the purpose of visualisation, we set $\sigma(t,x,y)=\sigma(t,x,x)$ for $x<y$.

\begin{figure}[!htp]
\includegraphics[width=\textwidth]{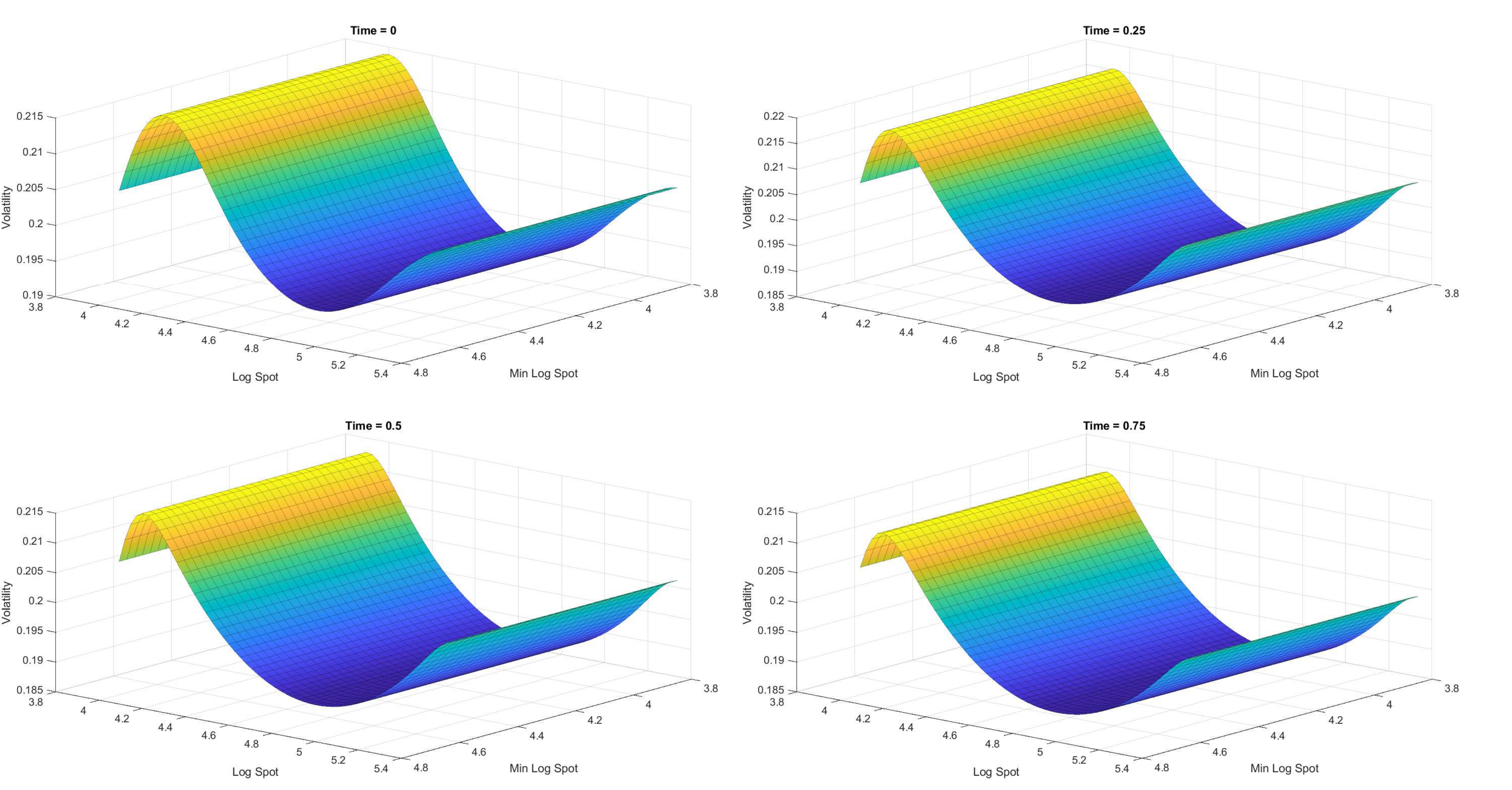}
\includegraphics[width=\textwidth]{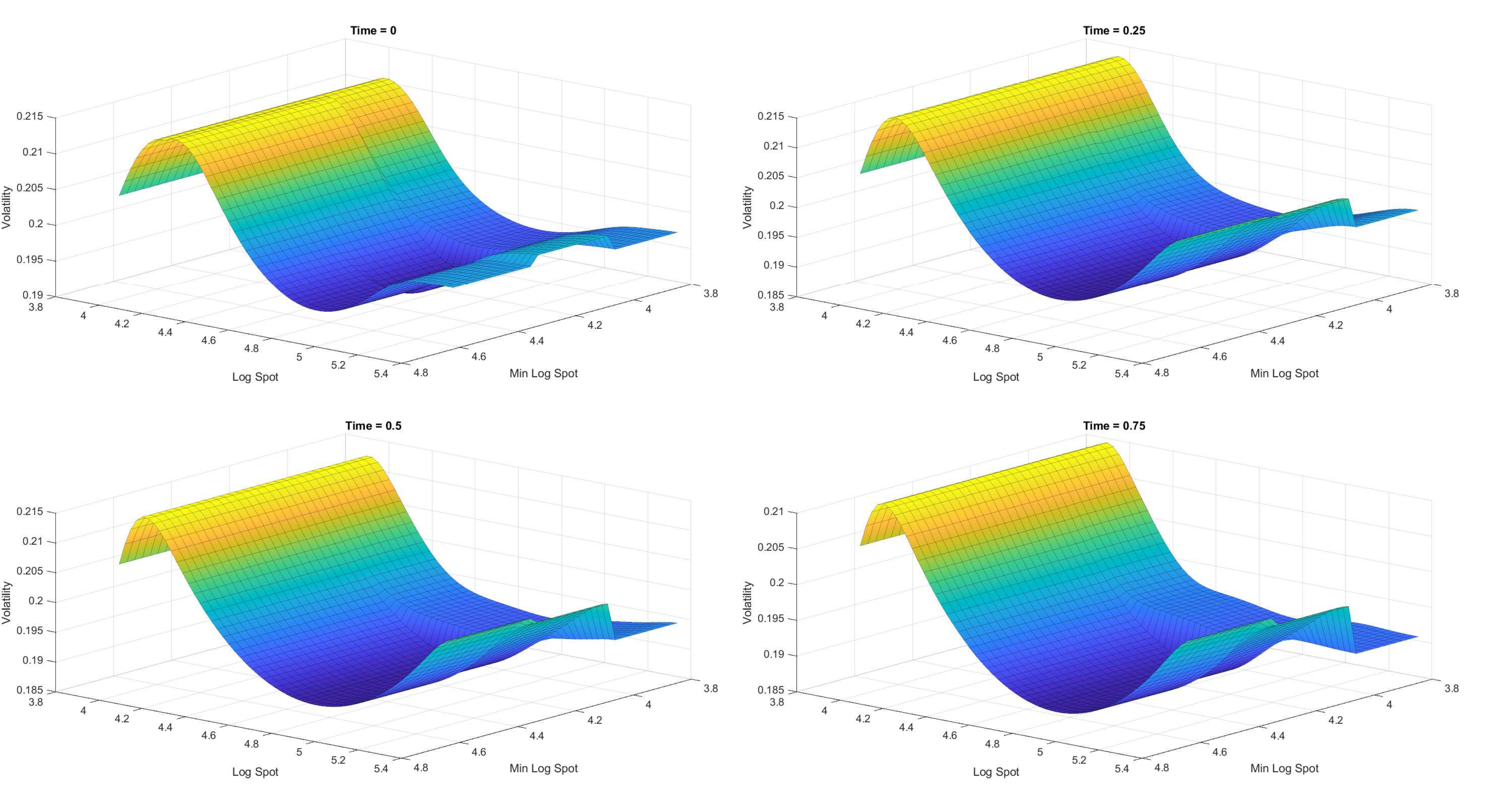}
\caption{The top half shows a volatility function $\sigma(t,x,y)$ calibrated to European options only. The bottom half is calibrated to European options and barrier options at two different barriers.}
\label{figlocalise}
\end{figure}

To further demonstrate the effect of dimension reduction, we repeat the same computation but removing some of the options. In the first test, we only calibrate to European options, while in the second test we calibrate to European options and barrier options at two different barriers $b_1<b_2<X_0$. The results are shown in Figure \ref{figlocalise}. When only European options are used, $\sigma$ only depends on $(t,x)$ but not $y$. In the cases where some barrier options are added, the dependence of $\sigma$ on $y$ can be divided into three regions, $y>b_2$, $b_2\geq y>b_1$ and $b_1\geq y$, corresponding to the number of barriers the underlying has hit. This behaviour is consistent with our dimension reduction results.

\appendix

\section{Appendix}

\subsection{Proof of Lemma \ref{lemmeasureconstraint1}}\label{appendix1}

\begin{proof}[Proof of Lemma \ref{lemmeasureconstraint1}]
The ``if'' direction follows immediately from the functional It\^o formula, so we will focus only on the ``only if'' direction. First note that we can translate $\phi$ by any $\phi(0,\cdot)\in C_b(\Omega_0)$ without altering the right hand side. This yields
\[
\int_{\Omega_0} \phi(0,\cdot) \, d(\mu \circ X^{-1}_0-\rho_0)=0
\]
for all $\phi(0,\cdot)\in C_b(\Omega_0)$. Therefore $\mu$ is a probability measure with $\mu\circ X_0^{-1}=\rho_0$.

Let $f\in C_b(\Lambda)$ be any function. Consider
$
\phi(t,X)=\int_0^t f(s,X)\, ds.
$
Then we have $\Dx \phi = \Dxx \phi = 0$ and $\Dt \phi = f$. Applying Fubini's theorem, we obtain
$
\int_{\Lambda} f \,(d\nu-d\hat\mu)  = 0,
$
Since $f$ is an arbitrary continuous function on $\Lambda$, this implies that $\nu=\hat\mu$ and we can rewrite our condition as
\begin{align}\label{eqidmeasureb01}
\E^\mu (\phi(1,X) - \phi(0,X)) = \E^\mu \int_0^1 \Dt \phi + \alpha_t \cdot \Dx \phi + \frac{1}{2}\beta_t : \Dxx\phi \, dt.
\end{align}

Fix $u\in[0,1]$ and let $I^n\in C^1_b([0,1])$ be a sequence of increasing functions with $I^n(t)=0$ for $t\in[0,u]$ and $I^n(t)=1$ for $t\in[u+1/n,1]$.
Define the function $P:\R^d\to \R$ by $P(x)=\sqrt{1+x\cdot x}\geq |x|_\infty$ and let $P^n\in C^2_b(\R^d)$ be a sequence of positive, bounded functions with uniformly bounded derivatives such that $P^n(x)=P(x)$ for $|x|_\infty\leq n$.
Set 
\[
\phi^n(t,X)=P^n(X_t-X_u)I^n(t).
\]
It is clear that $\phi^n\in C^{1,2}_0 (\Lambda)$. 
Then the integral of $\Dt\phi^n$ can be bounded by 
\begin{align*}
\left| \int_0^1 \Dt\phi^n\, dt\right| &= \left|\int_u^{u+\frac{1}{n}} P^n(X_t-X_{u})\partial_t I^n(t)\, dt \right|\\
&\leq \left(\max_{t\in[u,u+\frac{1}{n}]}| P^n(X_t-X_u)|\right) \int_u^{u+\frac{1}{n}}\partial_t I^n(t)\, dt\\
&= \max_{t\in[u,u+\frac{1}{n}]}|P^n(X_t-X_u)|,
\end{align*}
which converges to 1 as $n\to+\infty$.
The space derivatives of $\phi^n$ are uniformly bounded.
Using Fatou's lemma and the integrability of $(\alpha,\beta)$, we have
\[
 \E^\mu |X_1-X_u|_\infty\leq \E^\mu P(X_1-X_u)\leq \lim_{n\to+\infty}\E^\mu P^n(X_1-X_0) < +\infty,
\]
hence $X_1-X_u$ is $\mu$-integrable for all $u\in[0,1]$.

Next, fix $u\in[0,1]$ and let $K^n\in C^2_b(\R^d;\R^d)$ be a sequence of bounded functions with uniformly bounded derivatives satisfying 
$K^n(x)=x$ for $|x|_\infty\leq n$.
Let $g\in C_b(\Omega_{u};\R^d)$ be an arbitrary function and consider
\begin{align*}
\phi^n(t,X)&=K^n(X_t-X_{u})\cdot g(X_{\cdot\wedge u})I^n(t). 
\end{align*}
Using arguments similar to before, $\phi^n\in C^{1,2}_0 (\Lambda)$. Furthermore, the space derivatives of $\phi^n$ are uniformly bounded and
\begin{gather*}
\lim_{n\to+\infty}\int_0^1 \Dt\phi^n\, dt=0,\ \ \lim_{n\to+\infty}\Dx\phi^n=g(X_{\cdot\wedge u})\I(t>u),\ \  \lim_{n\to+\infty}\Dxx\phi^n=0.
\end{gather*}
By the integrability of $\alpha$ and $X_1-X_u$ as well as the dominated convergence theorem, as $n\to+\infty$,
\[
\E^\mu ((X_1-X_{u})\cdot g(X_{\cdot\wedge u})) = \E^\mu \bigg(\int_u^1 \alpha_t \, dt \cdot g(X_{\cdot\wedge u})\bigg).
\]
Recall that $g$ is arbitrary, which implies
\[
\E^\mu\bigg(X_1-X_{u}-\int_u^1 \alpha_t \, dt\,\bigg|\,\Filt_u\bigg)=0.
\]
Since this holds for all $u\in[0,1]$, and $\alpha, X_1-X_u$ are integrable, $M:=X-\int_0^\cdot \alpha_t \, dt$ must be a continuous $(\FF,\mu)$-martingale.

Applying the functional Ito's formula, our condition reduces to
\[
\E^\mu \bigg(\int_0^1 \Dxx\phi: (\beta\, dt-d\langle X\rangle_t)\bigg)  = 0.
\]
Fix $u\in[0,1]$, let $h\in C_b(\Omega_{u};\S^d)$ and consider
\[
\phi^n(t,X)=K^n(X_t-X_{u})^\intercal h(X_{\cdot\wedge u})K^n(X_t-X_{u}) I^n(t), 
\]
where $K^n$ and $I^n$ are defined as before. Once again, $\phi^n\in C^{1,2}_0 (\Lambda)$.
In particular, $\Dxx \phi^n$ takes value in $\S^d_+$, is uniformly bounded and satisfies
\[
\lim_{n\to+\infty}\Dxx \phi^n(t,X)= h(X_{\cdot\wedge u})\I(t>u).
\]
By the integrability of $\beta$, Fatou's lemma and the dominated convergence theorem, we have
\begin{gather*}
\E^\mu \int_u^1 h(X_{\cdot\wedge u}): d\langle X\rangle_t \leq \lim_{n\to+\infty} \E^\mu \int_0^1 \Dxx \phi^n: d\langle X\rangle_t\\
=\lim_{n\to+\infty} \E^\mu \int_0^1 \Dxx \phi^n: \beta\, dt=\E^\mu \int_u^1 h(X_{\cdot\wedge u}): \beta\, dt<+\infty,
\end{gather*}
Thus $\E^\mu \int_u^1 h(X_{\cdot\wedge u}): d\langle X\rangle_t <+\infty$ and we may apply the dominated convergence theorem again to change the first inequality to an equality, which yields
\[
\E^\mu \bigg(\int_u^1 h(X_{\cdot\wedge u}): (\beta\, dt-d\langle X\rangle_t) \bigg) = 0,
\]
for all $h\in C_0(\Omega_{u};\S^d)$.
Since $M=X-\int_0^\cdot \alpha_t \, dt$ is a martingale with $\langle M\rangle=\langle X\rangle$, we must have
\[
\E^\mu ((M_1-M_{u}) (M_1-M_{u})^\intercal \,|\, \Filt_u)=\E^\mu \bigg(\int_u^1 d\langle X\rangle_t \,\bigg|\, \Filt_u\bigg)=\E^\mu \bigg(\int_u^1 \beta_t \, dt \,\bigg|\, \Filt_u\bigg),
\]
and so $MM^\intercal-\int_0^\cdot \beta_t \, dt$ is an $(\FF,\mu)$-martingale. Therefore $\mu$ must be a semimartingale measure with characteristics $(\alpha,\beta)$, completing the proof.
\end{proof}

\subsection{Lemma \ref{lemcostfunciton01}}\label{appendix2}

\begin{lemma}\label{lemcostfunciton01}
Define $a:C_b(\Omega)\times C_b({\Omega_0})\times C_b({\Omega})\times C_b(\Lambda;\R\times\R^d\times \S^d)\to\R$ by
\begin{align*}
a(\psi,\bar\varphi,\varphi,p,q,r) &:= \begin{cases} \displaystyle
\int_{\Omega_0} \bar\varphi\, d\rho_0, & \text{if $\varphi\leq \epsilon$ and $p+H^*(q, r)\leq \epsilon$,} \\
+\infty, & \text{otherwise.}
\end{cases}
\end{align*}
Its convex conjugate $a^*: \Mset({\Omega})\times \Mset({\Omega_0})\times\Mset({\Omega})\times\Mset(\Lambda;\R\times\R^d\times \S^d)\to\R$ is given by
\begin{align*}
a^*(\xi,\rho,\mu,\nu,\bar\nu,\tilde\nu)&:= \begin{cases} \displaystyle
\int_{\Omega} \epsilon \,d\mu+\int_\Lambda \left(H\left(\frac{d\bar\nu}{d\nu},\frac{d\tilde\nu}{d\nu}\right)+\epsilon\right) d\nu, & \text{if $(\xi,\rho,\mu,\nu,\bar\nu,\tilde\nu)\in\AA$}, \\
+\infty, & \text{otherwise.}
\end{cases}
\end{align*}
where
\begin{align*}
\AA&:=\{(\xi,\rho,\mu,\nu,\bar\nu,\tilde\nu)\in \Mset({\Omega})\times \Mset({\Omega_0})\times\Mset({\Omega})\times\Mset(\Lambda;\R\times\R^d\times \S^d):\\
&\qquad\qquad \xi=0,\ \rho=\rho_0,\ \mu\geq 0,\ \nu \geq 0,\ (\bar\nu,\tilde\nu)\ll\nu\}.
\end{align*}
\end{lemma}
\begin{proof}
Throughout the proof, we will use the fact that $C_b$ is dense in $L^1$ with respect to the $L^1$ topology.
Let us identify the cases where $a^*<+\infty$.
Using the definition of convex conjugates, we have
\begin{align*}
a^*(\xi,\rho,\mu,\nu,\bar\nu,\tilde\nu)&=  \sup_{\substack{(\psi,\bar\varphi,\varphi,p,q,r)\\ \varphi\leq \epsilon, p+H^*(q, r)\leq \epsilon}} \int_{\Omega} (\psi\, d\xi+\varphi\, d\mu) + \int_{\Omega_0} \bar\varphi\, d(\rho-\rho_0)\\
&\qquad\qquad\qquad\qquad+\int_\Lambda (p\,d\nu +  q\cdot d\bar\nu +  r : d\tilde\nu).
\end{align*}
If $a^*<+\infty$, then $\xi=0$, $\rho=\rho_0$, $\mu\geq 0$ and $\nu \geq 0$. To see why one can restrict to $\mu\geq 0$, suppose $\mu(E)<0$ for some measurable set $E\subset\Omega$. Then there exists a sequence of non-positive functions $\varphi_n\in C_b(\Omega)$ that converge to $-\I(E)$ in $L^1(\Omega,\mu)$. By scaling $\varphi_n$ arbitrarily and adding them to $\varphi$, the function $a^*$ becomes unbounded. A similar argument shows that $\nu \geq 0$. So our function reduces to
\begin{align*}
a^*(\xi,\rho,\mu,\nu,\bar\nu,\tilde\nu)&=\int_\Omega \epsilon\, d\mu+ \sup_{p+H^*(q, r)\leq \epsilon} \int_\Lambda p\,d\nu +  q\cdot d\bar\nu +  r : d\tilde\nu.
\end{align*}

Next, since the function is linear in $(p,q,r)$, if $a^*$ is finite, then the supremum must occur at the boundary
\begin{align*}
a^*(\xi,\rho,\mu,\nu,\bar\nu,\tilde\nu)&= \int_\Omega \epsilon\, d\mu+ \sup_{p+H^*(q, r)= \epsilon} \int_\Lambda p\,d\nu +  q\cdot d\bar\nu +  r : d\tilde\nu\\
&= \int_\Omega \epsilon\, d\mu+\int_\Lambda \epsilon\, d\nu+ \sup_{\substack{(q, r)}} \int_\Lambda  q\cdot d\bar\nu +  r : d\tilde\nu-H^*(q, r)\,d\nu.
\end{align*}
We claim that it is necessary to have $(\bar\nu,\tilde\nu)\ll\nu$. Suppose that there exists a measurable set $E$ such that $(\bar\nu,\tilde\nu)(E)\neq 0$ but $\nu(E)=0$. Once again let us construct a sequence of continuous function in $C_b(\Lambda)$ converging to $\I(E)$ in $L^1(\Lambda)$ and add multiples of it (depending on the sign of $\nu(E)$) to $(q, r)$, which would allow $a^*$ to grow arbitrarily. Thus, we may write and bound $a^*$ in the following way,
\begin{align*}
a^*(\xi,\rho,\mu,\nu,\bar\nu,\tilde\nu)
&= \int_\Omega \epsilon\, d\mu+\int_\Lambda \epsilon\, d\nu+\sup_{(q, r)} \int_\Lambda \bigg(q\cdot \frac{d\bar\nu}{d\nu} +  r : \frac{d\tilde\nu}{d\nu}-H^*(q, r)\bigg)d\nu\\
&\leq \int_\Omega \epsilon\, d\mu+\int_\Lambda \epsilon\, d\nu+\int_\Lambda \sup_{(q, r)} \bigg(q\cdot \frac{d\bar\nu}{d\nu} +  r : \frac{d\tilde\nu}{d\nu}-H^*(q, r)\bigg)d\nu\\
&=\int_\Omega \epsilon\, d\mu+\int_\Lambda \epsilon\, d\nu+\int_\Lambda  H\bigg(\frac{d\bar\nu}{d\nu},\frac{d\tilde\nu}{d\nu}\bigg) d\nu.
\end{align*}
Note that we have used the lower-semicontinuity of $H$.
Equality can be shown by choosing $(q,r)_n$ to be a sequence of continuous functions converging to $\nabla H(\frac{d\bar\nu}{d\nu},\frac{d\tilde\nu}{d\nu})$, then applying the dominated convergence theorem and the fact that $H^*$ is continuous in $\operatorname{dom}(H^*)$.

Finally, we see that the conditions $\xi=0$, $\rho=\rho_0$, $\mu\geq 0$, $\nu \geq 0$ and $(\bar\nu,\tilde\nu)\ll\nu$ are necessary for $a^*<+\infty$. Therefore, the claim is proven.
\end{proof}

\bibliographystyle{acm}
\bibliography{bibfile}

\begin{thebibliography}{10}

\bibitem{AbergelTachet}
{\sc Abergel, F., and Tachet, R.}
\newblock A nonlinear partial integro-differential equation from mathematical
  finance.
\newblock {\em Discrete Contin. Dyn. Syst. 27}, 3 (2010), 907--917.

\bibitem{avellaneda1997calibrating}
{\sc Avellaneda, M., Friedman, C., Holmes, R., and Samperi, D.}
\newblock Calibrating volatility surfaces via relative-entropy minimization.
\newblock {\em Applied Mathematical Finance 4}, 1 (1997), 37--64.

\bibitem{benamou2000computational}
{\sc Benamou, J.-D., and Brenier, Y.}
\newblock A computational fluid mechanics solution to the {Monge-Kantorovich}
  mass transfer problem.
\newblock {\em Numerische Mathematik 84}, 3 (2000), 375--393.

\bibitem{Br5}
{\sc Brenier, Y.}
\newblock Minimal geodesics on groups of volume-preserving maps and generalized
  solutions of the {E}uler equations.
\newblock {\em Comm. Pure Appl. Math. 52}, 4 (1999), 411--452.

\bibitem{brezis1983analyse}
{\sc Br{\'e}zis, H.}
\newblock Analyse fonctionnelle, th{\'e}orie et applications,(1983), 1983.

\bibitem{brunick2013mimicking}
{\sc Brunick, G., Shreve, S., et~al.}
\newblock Mimicking an {It{\^o}} process by a solution of a stochastic
  differential equation.
\newblock {\em The Annals of Applied Probability 23}, 4 (2013), 1584--1628.

\bibitem{buckdahn2015pathwise}
{\sc Buckdahn, R., Ma, J., and Zhang, J.}
\newblock Pathwise {Taylor} expansions for random fields on multiple
  dimensional paths.
\newblock {\em Stochastic Processes and their Applications 125}, 7 (2015),
  2820--2855.

\bibitem{cont2013functional}
{\sc Cont, R., Fourni{\'e}, D.-A., et~al.}
\newblock Functional {It\^o} calculus and stochastic integral representation of
  martingales.
\newblock {\em The Annals of Probability 41}, 1 (2013), 109--133.

\bibitem{dolinsky2014martingale}
{\sc Dolinsky, Y., and Soner, H.~M.}
\newblock Martingale optimal transport and robust hedging in continuous time.
\newblock {\em Probability Theory and Related Fields 160}, 1-2 (2014),
  391--427.

\bibitem{dunford1958linear}
{\sc Dunford, N., and Schwartz, J.~T.}
\newblock {\em Linear operators part {I}: general theory}, vol.~7.
\newblock Interscience publishers New York, 1958.

\bibitem{dupire2009functional}
{\sc Dupire, B.}
\newblock Functional {It\^o} calculus.
\newblock {\em papers.ssrn.com\/} (2009).

\bibitem{dupire1994pricing}
{\sc Dupire, B., et~al.}
\newblock Pricing with a smile.
\newblock {\em Risk 7}, 1 (1994), 18--20.

\bibitem{ekren2016viscosity1}
{\sc Ekren, I., Touzi, N., Zhang, J., et~al.}
\newblock Viscosity solutions of fully nonlinear parabolic path dependent
  {PDEs}: Part {I}.
\newblock {\em The Annals of Probability 44}, 2 (2016), 1212--1253.

\bibitem{ekren2016viscosity2}
{\sc Ekren, I., Touzi, N., Zhang, J., et~al.}
\newblock Viscosity solutions of fully nonlinear parabolic path dependent
  {PDEs}: Part {II}.
\newblock {\em The Annals of Probability 44}, 4 (2016), 2507--2553.

\bibitem{fremlin1972bounded}
{\sc Fremlin, D., Garling, D., and Haydon, R.}
\newblock Bounded measures on topological spaces.
\newblock {\em Proceedings of the London Mathematical Society 3}, 1 (1972),
  115--136.

\bibitem{gobet2005regression}
{\sc Gobet, E., Lemor, J.-P., Warin, X., et~al.}
\newblock A regression-based {Monte Carlo} method to solve backward stochastic
  differential equations.
\newblock {\em The Annals of Applied Probability 15}, 3 (2005), 2172--2202.

\bibitem{guo2020joint}
{\sc Guo, I., Loeper, G., Obloj, J., and Wang, S.}
\newblock Joint modelling and calibration of spx and vix by optimal transport.
\newblock {\em arXiv preprint arXiv:2004.02198\/} (2020).

\bibitem{guo2019calibration}
{\sc Guo, I., Loeper, G., and Wang, S.}
\newblock Calibration of local-stochastic volatility models by optimal
  transport.
\newblock {\em arXiv preprint\/} (2019).

\bibitem{guo2019local}
{\sc Guo, I., Loeper, G., and Wang, S.}
\newblock Local volatility calibration by optimal transport.
\newblock In {\em 2017 MATRIX Annals}. Springer, 2019, pp.~51--64.

\bibitem{guyon2014path}
{\sc Guyon, J.}
\newblock Path-dependent volatility.
\newblock {\em https://ssrn.com/abstract=2425048\/} (2014).

\bibitem{LabTouz}
{\sc Henry-Labord\`ere, P., and Touzi, N.}
\newblock An explicit martingale version of the one-dimensional {B}renier
  theorem.
\newblock {\em Finance Stoch. 20}, 3 (2016), 635--668.

\bibitem{hou2018robust}
{\sc Hou, Z., and Ob{\l}{\'o}j, J.}
\newblock Robust pricing--hedging dualities in continuous time.
\newblock {\em Finance and Stochastics 22}, 3 (2018), 511--567.

\bibitem{huesmann2019benamou}
{\sc Huesmann, M., Trevisan, D., et~al.}
\newblock A benamou--brenier formulation of martingale optimal transport.
\newblock {\em Bernoulli 25}, 4A (2019), 2729--2757.

\bibitem{lecam1957convergence}
{\sc LeCam, L.}
\newblock Convergence in distribution of stochastic processes.
\newblock {\em Univ. California Publ. Statist 2\/} (1957), 207--236.

\bibitem{loeper2006reconstruction}
{\sc Loeper, G.}
\newblock The reconstruction problem for the {Euler-Poisson} system in
  cosmology.
\newblock {\em Archive for rational mechanics and analysis 179}, 2 (2006),
  153--216.

\bibitem{mikami2006duality}
{\sc Mikami, T., and Thieullen, M.}
\newblock Duality theorem for the stochastic optimal control problem.
\newblock {\em Stochastic processes and their applications 116}, 12 (2006),
  1815--1835.

\bibitem{pal2018exponentially}
{\sc Pal, S., Wong, T.-K.~L., et~al.}
\newblock Exponentially concave functions and a new information geometry.
\newblock {\em The Annals of Probability 46}, 2 (2018), 1070--1113.

\bibitem{privault2013stochastic}
{\sc Privault, N.}
\newblock {\em Stochastic finance: an introduction with market examples}.
\newblock Chapman and Hall/CRC, 2013.

\bibitem{ren2017convergence}
{\sc Ren, Z., and Tan, X.}
\newblock On the convergence of monotone schemes for path-dependent {PDEs}.
\newblock {\em Stochastic Processes and their Applications 127}, 6 (2017),
  1738--1762.

\bibitem{ren2017comparison}
{\sc Ren, Z., Touzi, N., and Zhang, J.}
\newblock Comparison of viscosity solutions of fully nonlinear degenerate
  parabolic path-dependent {PDEs}.
\newblock {\em SIAM Journal on Mathematical Analysis 49}, 5 (2017), 4093--4116.

\bibitem{rockafellar1966extension}
{\sc Rockafellar, R.~T., et~al.}
\newblock Extension of {Fenchel} duality theorem for convex functions.
\newblock {\em Duke mathematical journal 33}, 1 (1966), 81--89.

\bibitem{sentilles1972bounded}
{\sc Sentilles, F.~D.}
\newblock Bounded continuous functions on a completely regular space.
\newblock {\em Transactions of the American Mathematical Society 168\/} (1972),
  311--336.

\bibitem{tan2014discrete}
{\sc Tan, X., et~al.}
\newblock Discrete-time probabilistic approximation of path-dependent
  stochastic control problems.
\newblock {\em The Annals of Applied Probability 24}, 5 (2014), 1803--1834.

\bibitem{tan2013optimal}
{\sc Tan, X., Touzi, N., et~al.}
\newblock Optimal transportation under controlled stochastic dynamics.
\newblock {\em The annals of probability 41}, 5 (2013), 3201--3240.

\bibitem{veraguas2017martingale}
{\sc Veraguas, J.~B., Beiglb{\"o}ck, M., Huesmann, M., and K{\"a}llblad, S.}
\newblock Martingale {Benamou--Brenier}: a probabilistic perspective.
\newblock {\em arXiv preprint arXiv:1708.04869\/} (2017).

\bibitem{villani2003topics}
{\sc Villani, C.}
\newblock {\em Topics in optimal transportation}.
\newblock No.~58. American Mathematical Soc., 2003.

\bibitem{zhang2004numerical}
{\sc Zhang, J., et~al.}
\newblock A numerical scheme for {BSDEs}.
\newblock {\em The annals of applied probability 14}, 1 (2004), 459--488.

\bibitem{zhang2014monotone}
{\sc Zhang, J., and Zhuo, J.}
\newblock Monotone schemes for fully nonlinear parabolic path dependent {PDEs}.
\newblock {\em Journal of Financial Engineering 1}, 01 (2014), 1450005.

\end{thebibliography}

\end{document}